\documentclass[11pt,a4paper]{article}

\usepackage[margin=1in]{geometry}
\usepackage[T1]{fontenc}
\usepackage[utf8]{inputenc}
\usepackage{lmodern}
\usepackage{graphicx}
\usepackage{amssymb,amsmath,amsthm}
\usepackage{mathtools}
\usepackage{mathrsfs}
\usepackage{enumitem}
\usepackage{comment}
\usepackage[strict]{changepage} % wider margins for a few long displays only
\usepackage{bbm}
\usepackage{ragged2e}
\usepackage{float}
\usepackage{extarrows}
\usepackage{dsfont}
\usepackage{xcolor}
\usepackage{caption}
\usepackage[authoryear,longnamesfirst]{natbib}
\usepackage{authblk}
\usepackage[colorlinks=true,linkcolor=blue,citecolor=blue,urlcolor=blue]{hyperref}
\newcommand{\bulk}{\text{\normalfont bulk}}

\newcommand{\cond}{\omega}

\renewcommand{\P}{\mathbb{P}}
\newcommand{\mus }{\mu_{\rm{stat}}}
\newcommand{\1}{{\boldsymbol 1}}

\newcommand{\E}{\mathbb{E}}
\newcommand{\N}{\mathbb{N}}       % for natural numbers

\newcommand{\R}{\mathbb{R}}
\newcommand{\one}{\textbf{1}}
\newcommand{\dd}{\mathop{}\!\mathrm{d}}
\newtheorem{theorem}{Theorem}[section] 
\newtheorem{corollary}[theorem]{Corollary}
\newtheorem{proposition}[theorem]{Proposition}
\newtheorem{lemma}[theorem]{Lemma}

\theoremstyle{definition}

\theoremstyle{remark}
\newtheorem{remark}[theorem]{Remark}

\title{Non-equilibrium steady state of the symmetric exclusion process with reservoirs}
\author[1]{Jhon Astoquillca\thanks{Corresponding author.}}
\author[2]{Simone Floreani}
\author[3]{Adrián González-Casanova}
\affil[1]{Institute of Mathematics, Statistics and Computer Science, University of S\~ao Paulo, S\~ao Paulo, 05508-090, Brazil}
\affil[2]{Mathematical Institute, University of Oxford, Andrew Wiles Building, Radcliffe Observatory, Woodstock Road, Oxford OX2 6GG, United Kingdom}
\affil[3]{School of Mathematical and Statistical Sciences and the Center for Mechanisms of Evolution (Biodesign Institute), Arizona State University, 901 S Palm Walk, Tempe, AZ 85287, United States of America}
\date{}

\begin{document}
\maketitle

\begin{abstract}
Consider the open symmetric exclusion process on a connected general graph with vertexes in $[N-1]:=\{1,\ldots, N-1\}$ where points $1$ and $N-1$ are connected, respectively, to a left reservoir and a right reservoir with densities $\rho_L,\rho_R\in(0,1)$. We prove that the non-equilibrium steady state of such system is 
\[\mus = \sum_{I\subset \mathcal P([N-1]) }F(I)\bigg(\otimes_{x\in I}{\rm{Bernoulli}}(\rho_R)\otimes_{y\in [N-1]\setminus I}{\rm{Bernoulli}}(\rho_L) \bigg).\]
In the formula above $\mathcal P([N-1])$ denotes the power set of $[N-1]$ while the numbers $F(I)> 0$ are such that $\sum_{I\subset \mathcal P([N-1]) }F(I)=1$ and given in terms of absorption probabilities of the absorbing stochastic dual process. Via probabilistic arguments we compute explicitly the factors $F(I)$ when the graph is a homogeneous segment.
\end{abstract}


\section{Introduction}

\subsection{Boundary driven systems}
In the context of non-equilibrium statistical physics, a lot of attention has been devoted to the study of stationary properties of \textit{open} particle systems evolving on a finite graph (see, e.g., \cite{bertini_macroscopic_2007} and \cite{Mallik_2015} for an overview on the subject). The word open refers to the fact that the dynamics does not conserve the total number of particles due to an interaction with the external world, typically modelled via particle reservoirs (see, e.g., \cite{derrida_exact_1993-1} and  \cite{carinci_duality_2013-1} ) or,  in heat transfer, via thermal baths (see, e.g., \cite{KMP} and \cite{gilbert_heat_2017}): such systems are thus referred to as \textit{boundary driven}. Reservoirs are mechanisms that inject and remove particles from the system, imposing a fix density of particles at a given site of the graph. %When more than one reservoir imposing different densities values are connected to the graph, the system is said to be out of equilibrium: such condition is characterized by the presence of a non zero particle current at stationarity and the stationary measure of the system is referred to as \textit{non-equilibrium steady state}.
When multiple reservoirs, each imposing different density values, are connected to the graph, the system is considered to be out of equilibrium. This condition is characterized by the presence of a non-zero particle current at stationarity, and the stationary measure of the system is commonly referred to as the \textit{non-equilibrium steady state}. While for many closed (as opposite of open) particle systems, the stationary, actually reversible, measures are explicit and in product form, the action of the reservoirs destroy reversibility and long range correlations can emerge in the non-equilibrium steady state as shown in the seminal paper \cite{sphon88}. 
Finding explicit stationary measures for open systems is a key problem in statistical physics, and it continues to generate substantial interest (see, e.g., the recent works \cite{FranceschiniFrassekGiardina}, \cite{FrassekSEP} and \cite{FrassekGiardina}). 

For some systems, the celebrated \textit{matrix product  ansatz method} has been developed in \cite{derrida_exact_1993-1} to obtain in explicit form such long range correlations. This method works, for instance,  for  the open simple symmetric exclusion process (SSEP) on a one dimensional segment with sites  $\{1, \ldots, N-1\}$ where points $\{1, N-1\}$ are connected, respectively, to a left and a right reservoir. This is a system of simple symmetric random walks subject to the exclusion rule: only one particle per site is allowed and thus attempted jumps to occupied sites are suppressed. Moreover at sites $\{1, N-1\}$ particles are destroyed and created at specific given rates. After \cite{derrida_exact_1993-1}, other matrix and algebraic methods have been proposed to compute explicitly the correlations of the open SSEP (see \cite{schutzbook}, \cite{liggett_stochastic_1999} and \cite{FrassekSEP}) and the research around this system is  extremely active (see, e.g., \cite{Salez2023, gantert2020mixing,goncalves_non-equilibrium_2019, bertini_macroscopic_2015, landim_stationary_2006,derrida_entropy_2007}).

However, %as many of these cite paper assert, 
a probabilistic representation of the non-equilibrium steady state  of the open SSEP complementing the explicit correlations computed via the matrix ansatz  is still not available in the literature and the matrix computations performed in e.g. \cite{derrida_exact_1993-1} lack a probabilistic interpretation.

Moreover, matrix ansatz methods strongly rely on the fact that the underlying graph is a homogeneous segment and particles perform nearest neighbor jumps, while  there are several natural reasons  why one would like to overcome such limitation. First, many physical systems are not one dimensional and modelling the open SEP on a graph approximating a $d$-dimensional domain started to receive attention recently (see \cite{DelloSchiavoPortinaleSau}). Second, realistic models of particles should take care of the presence of spatial inhomogeneities (see, e.g., \cite{nandori_local_2016}) in the underlying media and these are often modelled with edge dependent weights. Third, extensive research has been conducted on the effects of symmetric long-range jumps in open exclusion processes, revealing interesting phenomena (see, e.g., \cite{bernardingonjimen}, \cite{bernardinjimen}, \cite{gonscotta} and  \cite{bernardin_frac}).%the effect of symmetric long range jumps in open exclusion processes have been extensively studied and interesting phenomena have been shown (see, e.g., \cite{bernardingonjimen}, \cite{bernardin_frac}).

The boundary driven symmetric exclusion process (SEP) on a general graph satisfies a property that turns out to be extremely useful: it is in stochastic duality relation with a dual particle system where particles evolve as in the primal model but the reservoirs are replaced by absorbing sites (see, e.g., \cite{carinci_duality_2013-1} and the recent work  \cite{schutz_2023_reverse}). Thus the dual system is conservative and if the graph is connected, all the particles   will be eventually absorbed. More precisely, this relation allow to compute the expected evolution of products of occupation variables of $n$ sites via the dual absorbing process with $n$ particles only. Similar approaches to related models have been considered in~\cite{GiardinaKurchan1,GiardinaKurchan2}.

In this paper we will use this relation to show that on a general graph with symmetric weights, the non-equilibrium steady state of the open SEP is a mixture measure of product of Bernoulli measures. Moreover we develop a probabilistic approach to derive the explicit formulas previously achieved by the matrix product ansatz and other algebraic methods.

\subsection{Boundary driven SEP and its stationary distribution}\label{section: results for SSEP}
Let $G$ be a finite connected graph with vertex set $[N-1]:=\{1,\ldots,N-1\}$ and edge set $E$. To each edge $\{x,y\}\in E$ we associate a symmetric weight $\omega_{x,y}\in (0,\infty)$ called conductance. We thus identify $G$ with the triple $([N-1], E, (\omega_{x,y})_{\{x,y\}\in E})$. The boundary driven SEP on $G$  with reservoirs parameters $\rho_L, \rho_R\in (0,1)$ and  $\omega_L, \omega_R>0$ is the Markov process $(\eta_t)_{t\ge 0}$ with state space $\mathcal X=\{0,1\}^{[N-1]}$ and infinitesimal generator
\begin{align}\label{eq: generator open SSEP}
{\mathcal L}  = {\mathcal L}^\bulk+\cond_L\, {\mathcal L}_L+\cond_R\,  {\mathcal L}_R
\end{align}
 where, for all bounded functions $f : {\mathcal X} \to \R$,
\begin{align*}
& {\mathcal L}^\bulk f(\eta) =\ \sum_{\{x,y\}\in E} \cond_{x,y} \left\{\begin{array}{c}
		\eta(x)\, (1- \eta(y)) \left(f(\eta^{x,y})-f(\eta)\right)\\[.15cm]
		+\,\eta(y)\, (1- \eta(x)) \left(f(\eta^{y,x})-f(\eta)\right)
	\end{array} \right\}\ , \\[2mm]
& \mathcal L_L f(\eta) = \eta(1)\, (1- \rho_L)\,	 (f(\eta^{1,-})-f(\eta))
	+  \rho_L\, (1- \eta(1))\, (f(\eta^{1,+})-f(\eta)) \quad \text{and} \\[2mm]
& \mathcal L_R f(\eta) = \eta(N-1)\, (1- \rho_R) \left(f(\eta^{N-1,-})-f(\eta)\right)
	+ \rho_R \left(1- \eta(N-1))\, (f(\eta^{N-1,+})-f(\eta)\right) \ .
\end{align*}

Above $\eta^{x,y}$ is the configuration in which a particle (if any) has been removed from $x \in [N-1]$ and moved at~$y \in [N-1]$, while $\eta^{x,-} \in \mathcal X$ is the configuration obtained from $\eta$ by destroying a particle (if present) from  site~$x \in [N-1]$ and  $\eta^{x,+} \in \mathcal X$  is the configuration obtained from $\eta$ by creating a particle (if not already present) at  site $x \in [N-1]$. In the above dynamics, the action of the reservoirs corresponds to the  $\cond_L\, {\mathcal L}_L+\cond_R\,  {\mathcal L}_R$ part of the generator. $\rho_L$  is the particle density imposed by the left  reservoir and  $\rho_R$ the one imposed by the right reservoir. $\omega_L$ is the conductance connecting the site $1$ to the fictitious point $0$ representing the  left  reservoir and $\omega_R$ is the conductance connecting the site $N-1$ to the fictitious point  $N$  representing the  right reservoirs.

It is well known that $(\eta_t)_{t\ge 0}$ is in stochastic duality relation with the Markovian interacting particle system  $(\xi_t)_{t\ge 0}$ with state space $\mathcal X \times \N_0^{\{0, N\}}$ and which evolves as the SEP on $G$ but the reservoirs are now replaced by absorbing sites~$\{0,N\}$. More precisely, in the dual system a particle at site $1$ can be absorbed at site $0$  at rate $\omega_L$  and a particle at site $N-1$ can be absorbed at site $N$ at rate $\omega_R$. We denote by $\hat \P_\xi$ the law of $(\xi_t)_{t\ge 0}$  starting from the configuration $\xi$ which can be identified with  a subset $I$ of vertices in $[N-1]\cup \{0,N\}$ via the relation $\xi(x)=\boldsymbol 1_{\{x\in I\}}$. Thus we write either $\hat \P_\xi$ or $\hat \P_I$. We refer the reader to Section \ref{section: dual boundary driven SSEP} below for the precise definition of    $(\xi_t)_{t\ge 0}$ and for the duality relation satisfied by the two processes. 

\begin{figure}
	\centering
	\includegraphics[width=0.8\linewidth]{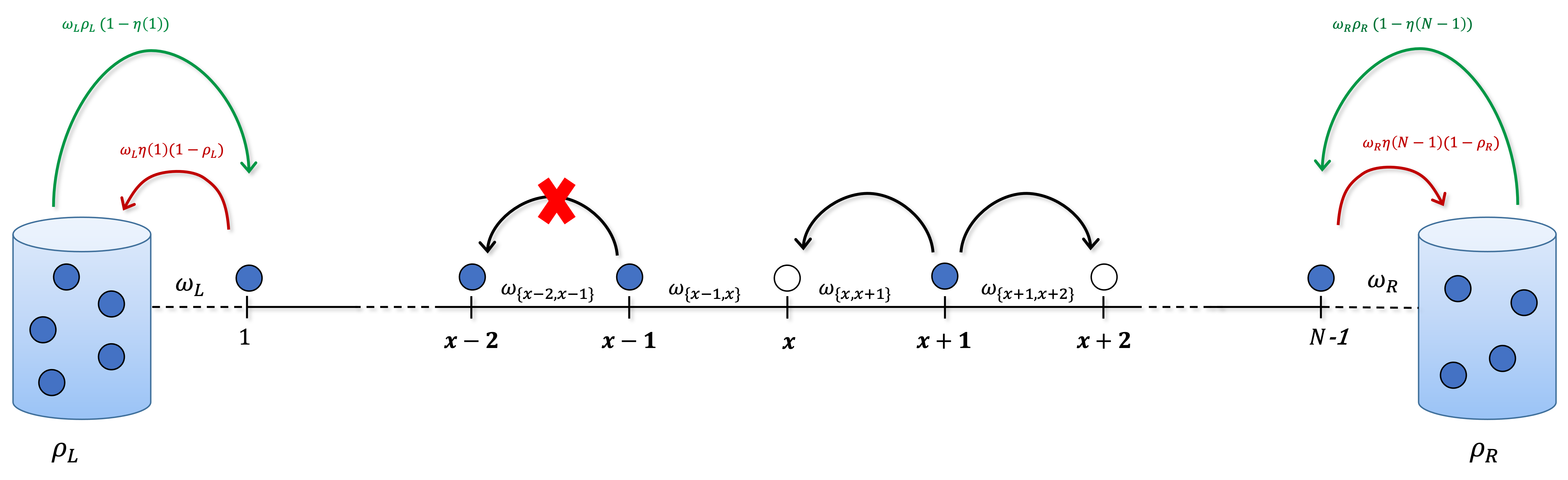}
	\caption{Schematic description of the dynamics.}
	\label{fig:bdssep}
\end{figure}

Our first main contribution is the following theorem. We denote by  $\mathcal P([N-1])$ the power set of $[N-1]$ and given~$I\in \mathcal P([N-1])$, $|I|$ denotes its cardinality.

\begin{theorem}\label{theorem: stationary measure SSEP}
%There exists positive weights  $(F(I))_{I\subset \mathcal P([N-1])}$ which depends only on the conductances $(\omega_{x,y})_{\{x,y\}\in E}$, $\omega_{0,1}$ and $\omega_{N-1,N}$ such that 
The stationary distribution of the boundary driven SEP on $G$ with reservoirs parameters~$\rho_L, \rho_R, \omega_L, \omega_R>0$ is 
\begin{align}\label{equation: stationary measure open SSEP}
\mus = \sum_{I\subset \mathcal P([N-1]) }F(I)\bigg(\otimes_{x\in I}{\rm{Bernoulli}}(\rho_R)\otimes_{y\in [N-1]\setminus I}{\rm{Bernoulli}}(\rho_L) \bigg),
\end{align}
where 
\begin{align}\label{equation: terms F(I) SSEP}
F(I):=\sum_{J\subset \mathcal P([N-1]): \ I\subset J } (-1)^{|J\setminus I|}\hat \P_J(\xi_\infty(N)=|J|)>0
\end{align}
satisfies 
$$ \sum_{I\subset \mathcal P([N-1]) }F(I)=1.$$
\end{theorem}
\begin{remark}[Probabilistic interpretation of $F(I)$]
	As it will be clear from Theorem \ref{Theorem: random bernulli} below, each factor $F(I)$ is the probability that all the particles initially at $I$ are absorbed at $N$, while the remaining ones at $0$, in the dual system of the boundary driven SEP, built via the labelled stirring construction and started from particles in each site of the bulk $[N-1]$.
\end{remark}

\begin{remark}
From \eqref{equation: stationary measure open SSEP} one can deduce that for all $\eta\in\mathcal X$,
$$\mus(\eta)=\sum_{\ell=0}^{|\eta|}\sum_{k=0}^{N-1-|\eta|}\rho_R^\ell\rho_L^{|\eta|-\ell}(1-\rho_R)^{k}(1-\rho_L)^{N-1-|\eta|-k}\left(\sum_{I\in C(\ell,k)}F(I) \right)$$
where $C(\ell,k)=\{I\in \mathcal P([N-1]): |I|=k+\ell \ {\rm{and}} \ |I\cap \eta|=\ell\}$ and $I\cap \eta$ denotes the set of points $x$ such that $x\in I$ and $\eta(x)=1$.
\end{remark}
\subsection{Explicit formulas for boundary driven SEP}
Providing explicitly the factors $F(I)$ amounts to compute the absorption probabilities in the dual system: via a coupling technique we compute  such probabilities  when the graph is a homogeneous one dimensional segment.

 \begin{theorem}\label{theorem: absorbtion probability SSEP conductances}
If \ $\omega_{x,y}=\one_{ \{|x-y|=1\}}$ for each $x,y\in[N-1]$ and  $\omega_L=\omega_R=1$, then, given $1\le x_1<\ldots <x_n\le N-1$,
 \begin{align}\label{eq: absorbtion probabilities SSEP conductance}
 	\hat \P_{\{x_1,\ldots, x_n\}}(\xi_\infty(N)=n) =\prod_{i=1}^n \frac{x_i-(i-1)}{N-(i-1)}.
 \end{align}
\end{theorem}

In several works (see, e.g., \cite{gantert2020mixing, goncalves_non-equilibrium_2019}) the following  choice of parameters is taken
\begin{equation*}
    \rho_L =\frac{\alpha}{\alpha + \gamma}, \quad \rho_R =\frac{\delta}{\delta+ \beta}, \quad \omega_L =\frac{1}{\alpha + \gamma}, \quad \omega_R =\frac{1}{\delta + \beta}
\end{equation*}
obtaining the $\alpha, \beta, \gamma, \delta$ model where particles are created at rate $\alpha$ on the left and at rate $\delta$ on the right while particles are destroyed at rate $\gamma$ on the left and $\beta$ on the right. In the setting of Theorem \ref{theorem: absorbtion probability SSEP conductances} but with general values of $\omega_L$ and $\omega_R$ as in the $\alpha, \beta, \gamma, \delta$ model we have that (see Lemma \ref{lemma: abs prob conductances at boundary} below), given $x_1<\ldots <x_n$,
\begin{align}\label{eq: absorbtion probabilities SSEP conductance at reservoirs}
	\hat \P_{\{x_1,\ldots, x_n\}}(\xi_\infty(N)=n) :=\prod_{i=1}^n \frac{\frac{1}{\omega_L}+x_i-i}{\frac{1}{\omega_L}+\frac{1}{\omega_R}+N-1-i}.
\end{align}
However we emphasize as $\rho_L$ and $\rho_R$ do not enter in the terms $F(I)$ defined in \eqref{equation: terms F(I) SSEP}.

Having obtained the probabilities that all the particles in the initial configuration are absorbed at $N$, it is then possible to recover all the absorption probabilities via the following relation.

\begin{proposition}\label{proposition: general absorption prob}
Given $1\le x_1<\ldots <x_n\le N-1$ and $0\le \ell\le n$

$$\hat \P_{\{x_1,\ldots, x_n\}}(\xi_\infty(N)=\ell)  =\sum_{k=\ell}^n (-1)^{k-\ell} {{k} \choose{\ell}}\sum_{1\le i_1<\ldots<i_k\le n}\hat \P_{\{x_{i_1},\ldots, x_{i_k}\}}(\xi_\infty(N)=k). $$
\end{proposition}
\begin{figure}
	\centering
	\includegraphics[width=0.8\linewidth]{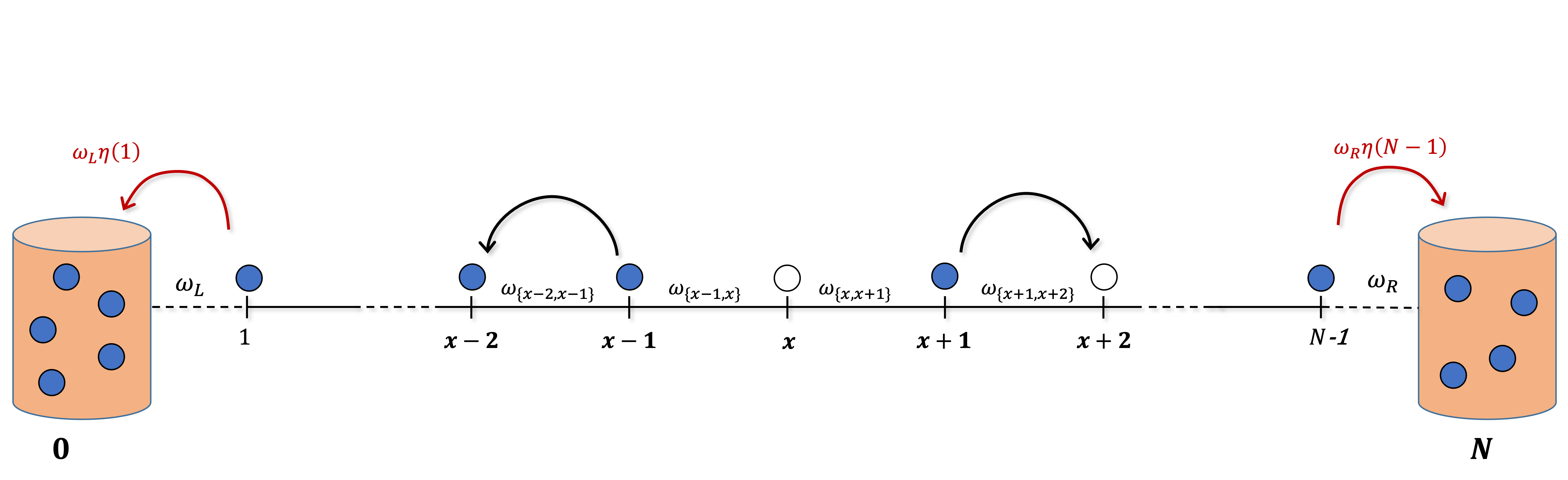}
	\caption{Schematic description of the dual dynamics.}
	\label{fig:dual bdssep}
\end{figure}

As a direct consequence of Theorem \ref{theorem: absorbtion probability SSEP conductances} and Proposition \ref{proposition: general absorption prob} we obtain the non-equilibrium $n$-point centered and non-centered correlations of the boundary driven SEP on the homogeneous segment.
\begin{proposition}[$n$-point non-equilibrium correlations]\label{proposition: n point correlation SSEP}
If \ $\omega_{x,y}=\one_{ \{|x-y|=1\}}$ for each $x,y\in[N-1]$ and  $\omega_L=\omega_R=1$, then given $n\in\N$ and $1\le x_1<\ldots<x_n\le N-1$, 
\begin{align}\label{eq: centered n points correlations}
\E_{\mus}\left[  \prod_{i=1}^n \left(  \eta(x_i)- \E[\eta(x_i)] \right)    \right]
=(\rho_R-\rho_L)^n \psi(x_1,\ldots,x_n)
\end{align}
where $\psi(x_1,\ldots,x_n)$ is given by
\begin{align}\label{eq: psi}
\sum_{j=0}^n(-1)^j\sum_{1\le i_1<\ldots <i_j\le n}  \   \prod_{\ell=1}^{j}\ \frac{x_{i_\ell}-(\ell-1)}{N-(\ell-1)} \prod_{r\in[n]\setminus\{i_1,\ldots,i_j\}}\frac{x_r}{N}
\end{align}
and 
\begin{align}\label{eq: non centered correlations}
	\E_{\mus}\left[  \prod_{i=1}^n  \eta(x_i) \right]
	=\sum_{j=0}^n\rho_L^{n-j}(\rho_R-\rho_L)^j \ \sum_{1\le i_1<\ldots <i_j\le n}  \   \prod_{\ell=1}^{j}\ \frac{x_{i_\ell}-(\ell-1)}{N-(\ell-1)}.
\end{align}
\end{proposition} 
\begin{remark}
For generic $\omega_L, \omega_R >0$, \eqref{eq: psi} above becomes 
$$\sum_{j=0}^n(-1)^j\sum_{1\le i_1<\ldots <i_j\le n}  \   \prod_{\ell=1}^{j}\ \frac{\frac{1}{\omega_L}+x_{i_\ell}-\ell}{\frac{1}{\omega_L}+\frac{1}{\omega_R}+N-1-\ell} \prod_{r\in[n]\setminus\{i_1,\ldots,i_j\}}\frac{\frac{1}{\omega_L}+x_r-1}{\frac{1}{\omega_L}+\frac{1}{\omega_R}+N-2} $$
and the right hand side of \eqref{eq: non centered correlations}
$$\sum_{j=0}^n\rho_L^{n-j}(\rho_R-\rho_L)^j \ \sum_{1\le i_1<\ldots <i_j\le n} \  \prod_{\ell=1}^{j}\ \frac{\frac{1}{\omega_L}+x_{i_\ell}-\ell}{\frac{1}{\omega_L}+\frac{1}{\omega_R}+N-1-\ell}.$$
\end{remark}

\subsection{Organization of the paper}
The rest of the paper is organized as follows. In Section \ref{section: stationary measure} we introduce properly the dual process of the boundary driven SEP. We then first express the stationary distribution as a product of Bernulli measures with random parameters (such result generalizes to the setting where more than two reservoirs are connected to $G$ as noted in Section \ref{section: generalization more than 2 res}) and from that  we prove Theorem \ref{theorem: stationary measure SSEP}. We also prove Proposition \ref{proposition: general absorption prob} and Proposition \ref{proposition: n point correlation SSEP} relying on Theorem \ref{theorem: absorbtion probability SSEP conductances} which is proved in the next sections. In Section \ref{section: 2 point correlations} we provide a simple probabilistic proof of the $2$ point correlations of the $\alpha, \beta, \gamma, \delta$ model by computing the absorption probabilities in a system with $2$ interacting particles. We also show how to compute the absorption probabilities of an arbitrary number of particles by induction after having made a guess of the expression. We conclude with Section \ref{section: Ninja method} where we provide a probabilistic coupling between the dual boundary driven SSEP on segments with different sizes to compute the absorption probabilities without the need to make an ansatz.

\section{The non-equilibrium steady state of the boundary driven SSEP}\label{section: stationary measure}
The main goal of this section is to prove Theorem \ref{theorem: stationary measure SSEP}.  As explained in the introduction our main technical tool is the stochastic duality relation satisfied by the boundary driven SEP that we are going to recall precisely below.
\subsection{The dual process and the duality relation}\label{section: dual boundary driven SSEP}
Consider the graph  $$G=([N-1], E, (\omega_{x,y})_{\{x,y\}\in E})$$ introduced in Section \ref{section: results for SSEP},  denote by $[N]_0=[N-1]\cup\{0,N\}$ and put  $$\hat E= E \cup\{\{0,1\},\{N-1,N\}\}$$  $\omega_{0,1}=\omega_L$ and $\omega_{N-1,N}=\omega_R$. The dual process $(\xi_t)_{t\ge 0}$ is the Markovian interacting particle system evolving on the extended graph $$\hat G=([N]_0, \hat E, (\omega_{x,y})_{\{x,y\}\in \hat E}),$$
which  behaves in the same way as the boundary driven SEP in the bulk $[N-1]$ but the reservoirs are now substituted by the absorbing sites $\{0, N\}$. Thus $\xi=(\xi_t)_{t\ge 0}$ has state space $\hat{ \mathcal X}:=\N^{\{0\}}\times \{0,1\}^{[N-1]} \times\N^{\{N\}}$ and its generator is given by
\begin{align}\label{eq: generator dulaSSEP}
	\hat{\mathcal L} = \hat{\mathcal L}^{bulk}+\cond_L\, \hat{\mathcal L}_L+\cond_R\,  \hat{\mathcal L}_R
\end{align}
where, for all bounded functions $f : \hat{\mathcal X} \to \R$,
\begin{align*}
		\hat{\mathcal L}^\bulk f(\xi) =\ \sum_{\{x,y\}\in E} \cond_{x,y} \left\{\begin{array}{c}
		\xi(x)\, (1- \xi(y)) \left(f(\xi^{x,y})-f(\xi)\right)\\[.15cm]
		+\,\xi(y)\, (1- \xi(x)) \left(f(\xi^{y,x})-f(\xi)\right)
	\end{array} \right\}\ ,
\end{align*}
\begin{align*}
	\nonumber
		\hat{\mathcal L}_L f(\xi) =&\  \xi(1)\,  (f(\xi^{1,0})-f(\xi))
\end{align*}
and
\begin{align*}
	\nonumber
	\hat{\mathcal L}_R f(\xi) =&\  \xi(N-1)\,  \left(f(\xi^{N-1,N})-f(\xi)\right).
\end{align*}
Recall that a configuration $\xi\in \mathcal {\hat X}$ will be often identified with the set $I\subset [N]_0$ of points $x$ such that $\xi(x)>0$, i.e.  $\xi(x)=\boldsymbol 1_{\{x\in I\}}$,   %For each $\xi\in \hat {\mathcal X}$ and the corresponding set $I\subset [N]_0$, 
and  that we denote by $\hat \P_\xi$ (or by $\hat \P_I$) and by $\hat \E_\xi$ (or by $\hat\E_I$) the law and the corresponding expectation of the process $(\xi_t)_{t\ge 0}$ starting from the configuration  $\xi$.

\smallskip
For $\eta\in\mathcal X$ and  $\xi\in\hat {\mathcal X}$ we define
$$D(\eta,\xi)=\rho_L^{\xi(0)} \left(\prod_{x\in[N-1]: \ \xi(x)=1}\eta(x)\right)\rho_R^{\xi(N)}.$$
%We recall the following classical duality relation below.
We recall below the duality relation between the boundary driven SEP and the process $(\xi_t)_{t\ge 0}$.
\begin{proposition}[See \cite{carinci_duality_2013-1} and \cite{FRS_ort_dual}]\label{proposition: duality}
The boundary driven SEP $(\eta_t)_{t\ge 0}$ with generator given in \eqref{eq: generator open SSEP} and the process $(\xi_t)_{t\ge 0}$ with generatore given in \eqref{eq: generator dulaSSEP} satisfy the following relation: for all $\eta\in\mathcal X$, $\xi\in\hat {\mathcal X}$ and $t\ge 0$
\begin{align}
\E_\eta[D(\eta_t, \xi)]=\hat \E_\xi[D(\eta, \xi_t)].
\end{align}
\end{proposition}

\subsection{The labelled stirring dual process}
In order to prove  Theorem \ref{theorem: stationary measure SSEP} we consider the labelled  stirring construction  (see, e.g., \cite{DeMasiPresutti} or \cite{liggett_interacting_2005-1}) of the dual process on $\hat G$. More precisely,  define on the same probability space independent random variables $(\Gamma_{x,y})_{\{x,y\}\in E}$, $\Gamma_0$, $\Gamma_{N}$ where $\Gamma_{x,y}$ is a Poisson point process with rate  $\omega_{x,y}$ providing the swapping times across the edge $\{x,y\}$ and $\Gamma_0$ and $\Gamma_{N}$  are  Poisson processes with rate $\omega_L$ and $\omega_R$ respectively, giving the absorbing times from sites  $1$  and $N-1$ respectively.   The  labelled dual particles  system $(X^{x_1}_t,\ldots, X^{x_n}_t)_{t\ge 0}$ starting from $x_1<\ldots<x_n$ is obtained by following deterministically the arrows obtained  as a realization of the Poisson processes just introduced. Namely $(X^{x_1}_t,\ldots, X^{x_n}_t)_{t\ge 0}$ is constant until a Poisson mark is encountered. If $t\in \Gamma_{x,y}$ and if at $t^-$ there is a  particle at $x$ or at $y$ or at both sites, then these particles swap their positions (i.e. if a particle is at $x$ at $t^-$ then it moves to $y$  at time $t$ and viceversa), keeping track of their label given by the initial position. If $t\in \Gamma_0$ (or $t\in \Gamma_{N}$) and at $t^-$ a  particle is present at $1$ (or at $N-1$) then that particles is absorbed at $0$ (or at $N$).
%Namely $(X^{x_1}_t,\ldots, X^{x_n}_t)_{t\ge 0}$ is constant outside $\Gamma_0\cup\Gamma_1\cup_{\{x,y\}\in E}\Gamma_{x,y}$, for $t\in \Gamma_{x,y}$, if there are  particles at $x$ or $y$ or at both sites, then these particles swap their positions (i.e. if a particle is at $x$ then it moves to $y$ and viceversa), keeping track of their label given by the initial position. If $t\in \Gamma_0$ (or $t\in \Gamma_{N-1}$) and a particle is present at $1$ (or at $N-1$) then that particles is absorbed at $0$ (or at $N$).

For each $n\in [N-1]$ and $x_1<\ldots<x_n$, we denote by $\P^{stir}_\xi$ (respectively $\E^{stir}_\xi$) the probability (respectively the expectation) induced by the random arrows on the space of trajectories of the labelled stirring process started from $\xi=\{x_1,\ldots,x_n\}$.

Notice that by the  construction it follows that for al $\Gamma \subset  [N]$ and $f:[N]_0^\Gamma\to \R$
\begin{align}\label{eq: consistency dual SSEP}
	\E^{stirr}_{[N]}[  f( (X^x_t)_{x\in \Gamma})]= \E^{stirr}_{\Gamma}[  f( (X^x_t)_{x\in \Gamma})].
	\end{align}
Moreover, if $f:[N]_0^\Gamma\to \R$ is symmetric (where symmetric means invariant by permutations of its entries), then 
\begin{align}\label{eq: equality in law of stirring and configuration dual SSEP}
	\E^{stirr}_{\Gamma}[  f( (X^x_t)_{x\in \Gamma})]=\hat \E_\Gamma[f(\xi_t) ].
	\end{align}

\subsection{Probabilistic interpretation of the non-equilibrium steady state and proof of Theorem \ref{theorem: stationary measure SSEP}}\label{section: proof of theorem stationary measure SSEP}

 We  introduce random variables $(U_x)_{x\in[N-1]}$ on the same probability space of the stirring construction, with $U_x\in\{\rho_L, \rho_R\}$ for all $x\in [N-1]$. Such random variables are not independent and  their  joint law  is given by
\begin{align*}
&\P_{[N-1]}^{stirr}((U_x)_{x\in[N-1]}=\bar \rho_I)= \E^{stirr}_{[N-1]} \left[\prod_{x\in I}\boldsymbol 1_{\{X^x_\infty=N\}} \prod_{x\in [N-1]\setminus I}\boldsymbol 1_{\{X^x_\infty=0\}}\right]
\end{align*}
where, given $I\subset [N-1]$, $\bar \rho_I$ is the vector $(\bar \rho_I(x))_{x\in [N-1]}$ with $\bar \rho_I(x)=\rho_R$ if $x\in I$ and  $\bar \rho_I(x)=\rho_L$ otherwise.

 We start by showing  the following result, which clarifies the probabilistic interpretation of the non-equilibrium steady state.
 \begin{theorem}\label{Theorem: random bernulli}The stationary distribution of the boundary driven SEP on $G$ with reservoirs parameters $\rho_L, \rho_R, \omega_L, \omega_R>0$ is 
\begin{align}\label{eq: stat measue SSEP in terms of rv U}
\mus=\E^{stirr}_{[N-1]}\left[\otimes_{x\in [N-1]} {\rm{Ber}}(U_x)\right].
\end{align}
\end{theorem}
\begin{proof}
The stationary measure $\mus$ is completely characterized by the moments $$\E_{\mus}\left[\prod_{i=1}^n\eta(x_i)\right], \ \forall n\in[N-1] \ \text{and } \ 1\le x_1<\ldots<x_n\le N-1.$$ By duality, we have that for all $\eta\in \mathcal X$
\begin{align}\label{eq: duality computation}
	\nonumber \E_{\mus}\left[\prod_{i=1}^n\eta(x_i)\right]=\lim_{t\to \infty}\E_\eta\left[\prod_{i=1}^n\eta_t(x_i)\right] 
	\nonumber&=\lim_{t\to \infty}\hat \E_{\{x_1,\ldots,x_n\}}[D(\eta, \xi_t)]=\hat \E_{\{x_1,\ldots,x_n\}}\left[    \rho_L^{\xi_\infty(0)}  \rho_R^{\xi_\infty(N)}     \right]\\
	&=\sum_{\ell=0}^n \rho_R^\ell \rho_L^{n-\ell}\hat\P_{\{x_1,\ldots,x_n\}}[\xi_\infty(N)=\ell].
\end{align}
Thus we only need to show that for all $n\in[N-1]$  and  $1\le x_1<\ldots< x_n\le N-1$
\begin{align}\label{eq: desired equality proof labelled SSEP}
\int \prod_{i=1}^n\eta(x_i)\dd \E^{stirr}_{[N-1]}\left[\otimes_{x\in [N-1]} {\rm{Ber}}(U_x)\right]= \sum_{\ell=0}^n \rho_R^\ell \rho_L^{n-\ell}\hat\P_{\{x_1,\ldots,x_n\}}[\xi_\infty(N)=\ell].
\end{align}

\begin{figure}
	\centering
	\includegraphics[width=0.6\linewidth]{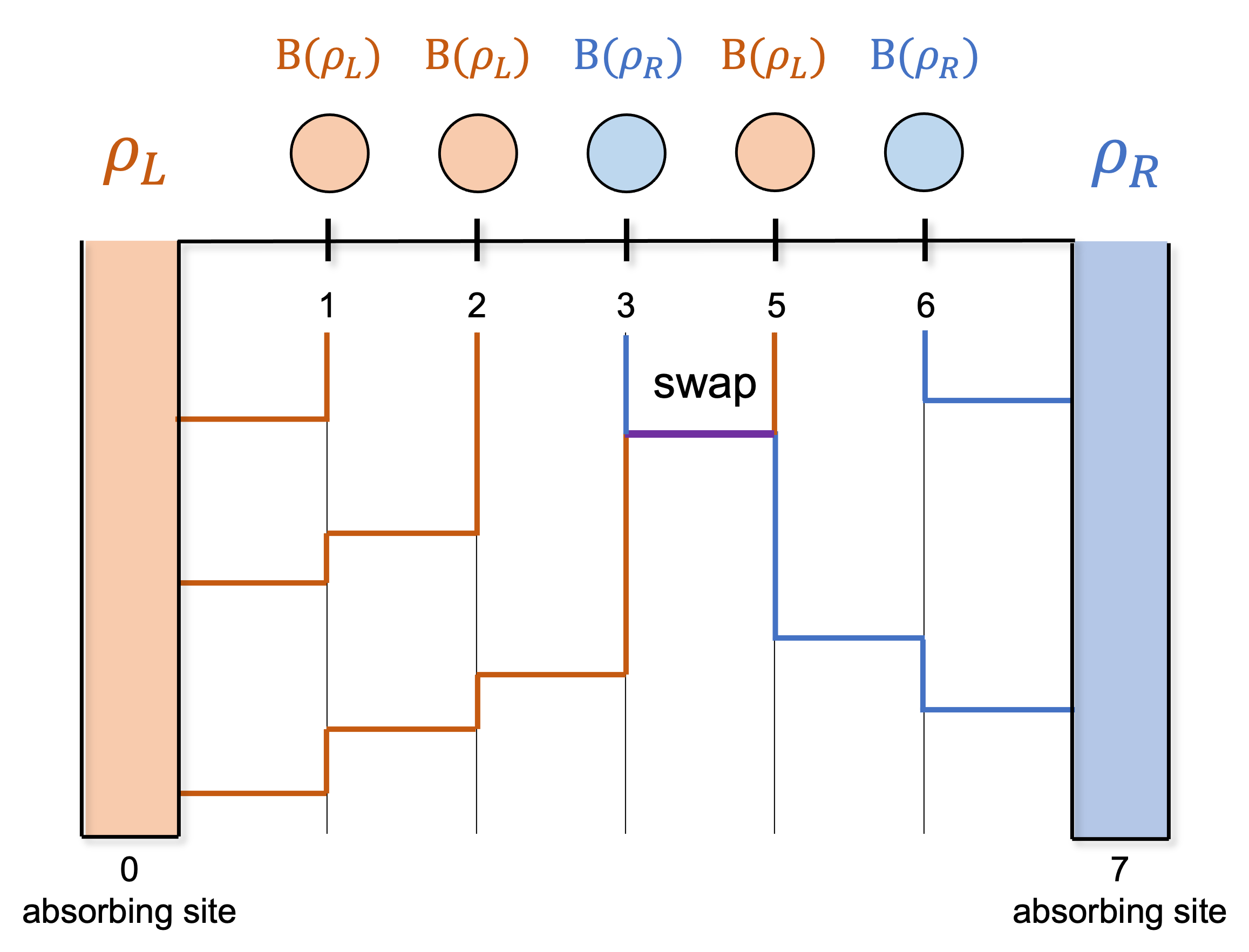}
	\caption{Schematic description of the proof via the stirring construction of the dual process.}
	\label{fig:ninjaparticle7}
\end{figure}
We then  have, using \eqref{eq: consistency dual SSEP},
\begin{align*}
\int \prod_{i=1}^n\eta(x_i)\dd & \E^{stirr}_{[N-1]} \left[\otimes_{x\in [N-1]} {\rm{Ber}}(U_x)\right]= \E^{stirr}_{\{x_1,\ldots,x_n\}}\left[\prod_{i=1}^n U_{x_i}\right]\\
&=\sum_{I\subset \{x_1,\ldots,x_n\}}\rho_R^{|I|}\rho_L^{n-|I|} \E^{stirr}_{\{x_1,\ldots,x_n\}}\ \left [\prod_{x\in I}\boldsymbol 1_{\{X^x_\infty=N\}} \prod_{x\in \{x_1,\ldots,x_n\}\setminus I}\boldsymbol 1_{\{X^x_\infty=0\}}\right]\\
&=\sum_{\ell=0}^n\rho_R^\ell \rho_L^{n-\ell}    \E^{stirr}_{\{x_1,\ldots,x_n\}}\ \left[ \sum_{I\subset \{x_1,\ldots,x_n\}\ : \ |I|=\ell}\left(\prod_{x\in I}\boldsymbol 1_{\{X^x_\infty=N\}} \prod_{x\in \{x_1,\ldots,x_n\}\setminus I}\boldsymbol 1_{\{X^x_\infty=0\}}\right)\right].
\end{align*}
Noticing that $$\sum_{I\subset \{x_1,\ldots,x_n\}\ : \ |I|=\ell}\left(\prod_{x\in I}\boldsymbol 1_{\{X^x_\infty=N\}} \prod_{x\in \{x_1,\ldots,x_n\}\setminus I}\boldsymbol 1_{\{X^x_\infty=0\}}\right)=\boldsymbol 1_{\{\sum_{\ell=1}^n\1_{\{X^{x_\ell}_\infty=N\}}=\ell \}}$$
is a symmetric function of $(X_\infty^{x_1},\ldots, X_\infty^{x_n})$ and using \eqref{eq: equality in law of stirring and configuration dual SSEP}, we have that 
$$\E^{stirr}_{\{x_1,\ldots,x_n\}} \left[ \sum_{I\subset \{x_1,\ldots,x_n\}\ : \ |I|=\ell}\left(\prod_{x\in I}\boldsymbol 1_{\{X^x_\infty=N\}} \prod_{x\in \{x_1,\ldots,x_n\}\setminus I}\boldsymbol 1_{\{X^x_\infty=0\}}\right)\right] =\hat \P_{\{x_1,\ldots,x_n\}\ }\left( \xi_\infty(N)=\ell\right)$$
from which \eqref{eq: desired equality proof labelled SSEP} follows and in particular \eqref{eq: stat measue SSEP in terms of rv U} also follows.
\end{proof}
We can finally obtain Theorem \ref{theorem: stationary measure SSEP}.
\begin{proof}[Proof of Theorem \ref{theorem: stationary measure SSEP}]
For any $I\subset [N-1]$
\begin{align*}
\P^{stirr}_{[N-1]}\left( \cup_{x\in I} \{U_x=\rho_R\}\cup_{y\in [N-1]\setminus I} \{U_y=\rho_L\} \right) & =\E^{stirr}_{[N-1]} \left[ \prod_{x\in I}\boldsymbol 1_{\{X^x_\infty=N\}} \prod_{x\in \{x_1,\ldots,x_n\}\setminus I}\boldsymbol 1_{\{X^x_\infty=0\}}\right]\\
&= \E^{stirr}_{ [N-1]} \left [\prod_{x\in I}\boldsymbol 1_{\{X^x_\infty=N\}} \prod_{x\in [N-1]\setminus I}\left(1-\boldsymbol 1_{\{X^x_\infty=N\}}\right)\right]\\
&= \sum_{J\subset [N-1]\setminus I}(-1)^{|J|}\E^{stirr}_{I\cup J} \left [ \prod_{x\in I\cup J}\boldsymbol 1_{\{X^x_\infty=N\}}\right]
\end{align*}
where we used \eqref{eq: consistency dual SSEP} in the last equality. Noticing that $\prod_{x\in I\cup J}\boldsymbol 1_{\{X^x_\infty=N\}}$ is a symmetric function of $(X^x_\infty)_{x \in I\cup J}$, using \eqref{eq: equality in law of stirring and configuration dual SSEP}, we conclude that 
$$\P^{stirr}_{[N-1]}\left( \cup_{x\in I} \{U_x=\rho_R\}\cup_{y\in [N-1]\setminus I} \{U_y=\rho_L\} \right) = \sum_{J\subset [N-1]\setminus I}(-1)^{|J|}\hat \P_{I\cup J} \left(\xi_\infty(N)=|I\cup J|\right) $$
from which we obtain
$$\E^{stirr}\left[\otimes_{x\in [N-1]} {\rm{Ber}}(U_x)\right]=\sum_{J\subset \mathcal P([N-1]): \ I\subset J }F(I)\bigg(\otimes_{x\in I}{\rm{Ber}}(\rho_R)\otimes_{x\in [N-1]\setminus I}{\rm}{Ber}(\rho_L) \bigg),$$
with $F(I)$ given in \eqref{equation: terms F(I) SSEP} and thus the thesis of Theorem \ref{theorem: stationary measure SSEP}. 

\end{proof}

\subsection{Generalization of Theorem \ref{Theorem: random bernulli}}\label{section: generalization more than 2 res}
Theorem \ref{Theorem: random bernulli} above extends easily to a boundary driven SEP on $G$ in which more than two sites are connected to corresponding reservoirs (such setting  being considered in, e.g., \cite{DelloSchiavoPortinaleSau} and \cite{Salez2023}). Indeed let $V'\subset V$ be the set of points of the graph connected to reservoirs and let $(\rho(x))_{x\in V'}$, with $\rho(x)\in(0,1)$, be the associated reservoirs densities. Consider the particles system $(\eta_t)_{t\ge 0}$
with state space $\mathcal X=\{0,1\}^{[N-1]}$ and infinitesimal generator
\begin{align}\label{eq: generator general open SEP}
	{\mathcal L}  = {\mathcal L}^\bulk+{\mathcal L}^{\rm{res}}
\end{align}
where ${\mathcal L}^\bulk$ is given in \eqref{eq: generator open SSEP}, while, for all bounded functions $f : {\mathcal X} \to \R$,
\begin{align*}
	{\mathcal L}^{\rm{res}} f(\eta) = \ \sum_{x\in V'} \cond_{x}^{\rm{res}} \left\{\begin{array}{c}
		\eta(x)\, (1- \rho(x)) \left(f(\eta^{x,-})-f(\eta)\right)\\[.15cm]
		+\,\rho(x)\, (1- \eta(x)) \left(f(\eta^{x,+})-f(\eta)\right)
	\end{array} \right\}\ ,
\end{align*}
where $\cond_{x}^{\rm{res}}>0$ is the conductance connecting the site $x\in V'$ to the corresponding reservoir. As stated in \cite[Remark 2.2]{FRS_ort_dual}, the duality relation given in Proposition \ref{proposition: duality}, extends to this more general setting but now, in the dual process, each vertex in $V'$ is connected to a corresponding absorbing site.

\begin{comment}
As stated in \cite[Remark 2.2]{FRS_ort_dual}, the duality relation given in Proposition \ref{proposition: duality}, extends to this more general setting.  More precisely, in the dual process, each vertex $x$ in $V'$ is connected to a corresponding absorbing site $\hat x$ and the duality function is given by 
$$D(\eta,\xi)=\left(\prod_{x\in[N-1]: \ \xi(x)=1}\eta(x)\right) \left( \prod_{\hat x: \ x\in V'} \rho(x)^{\xi(x')}\right).$$
The dual process $\xi=(\xi_t)_{t\ge 0}$ has state space $\hat{ \mathcal X}:=\left(\prod_{x\in V}\N^{\{\hat x\}}\right) \times \{0,1\}^{[N-1]}$ and its generator is given by
\begin{align}
\hat{\mathcal L} = \hat{\mathcal L}^{bulk}+\hat {\mathcal L}^{\rm{res}}
\end{align}
where $\hat{\mathcal L}^{bulk}$ is the same as in \eqref{eq: generator dulaSSEP}, while for all bounded functions $f : \hat{\mathcal X} \to \R$,
\begin{align*}
\hat {\mathcal L}^{\rm{res}} f(\eta) = \ \sum_{x\in V'} \cond_{x}^{\rm{res}} 
\xi(x) \left(f(\xi^{x,\hat x})-f(\xi)\right).
\end{align*}
\end{comment}
Denote again by 
$$(X^1_t,\ldots, X^{[N-1]}_t)_{t\ge 0}$$
the labelled stirring dual process of the process with generator given in \eqref{eq: generator general open SEP} and by $\E^{\rm{stirr}}_{[N-1]}$ its expectation. We then have the following theorem, which generalizes Theorem \ref{Theorem: random bernulli} above and whose proof is omitted being a straightforward generalization of the proof of Theorem \ref{Theorem: random bernulli}.

 \begin{theorem}\label{Theorem: random bernulli general}The stationary distribution of the boundary driven SEP on $G$ where sites in $V'\subset V$ are connected to  reservoirs with parameters $(\rho(x))_{x\in V'}$ and $(\omega_x^{\rm{res}})_{x\in V'}$  is 
	\begin{align}
	\mus=\E^{stirr}_{[N-1]}\left[\otimes_{x\in [N-1]} {\rm{Ber}}(\rho(X^x_{\infty}))\right].
\end{align}
\end{theorem}

\subsection{Proofs of Proposition \ref{proposition: general absorption prob} and Proposition \ref{proposition: n point correlation SSEP}}
We conclude the section by proving Proposition \ref{proposition: general absorption prob} and Proposition \ref{proposition: n point correlation SSEP} which are also achieved by duality.
\begin{proof}[Proof of Proposition \ref{proposition: general absorption prob}]
	Arguing as in the Proof of Theorem \ref{theorem: stationary measure SSEP}, we obtain
	\begin{adjustwidth}{-0.35in}{-0.35in}
\begin{align*}
		\hat \P_{\{x_1,\ldots, x_n\}}(\xi_\infty(N)=\ell) & =\sum_{I\subset \{x_1,\ldots,x_n\} \ : \ |I|=\ell}\quad  \E^{stirr}_{\{x_1,\ldots,x_n\}}\ \left [\prod_{x\in I}\boldsymbol 1_{\{X^x_\infty=N\}} \prod_{x\in \{x_1,\ldots,x_n\}\setminus I}\boldsymbol 1_{\{X^x_\infty=0\}}\right]\\
		&=\sum_{I\subset \{x_1,\ldots,x_n\} \ : \ |I|=\ell}\quad  \E^{stirr}_{\{x_1,\ldots,x_n\}}\ \left [\prod_{x\in I}\boldsymbol 1_{\{X^x_\infty=N\}} \prod_{x\in \{x_1,\ldots,x_n\}\setminus I}\left(1-\boldsymbol 1_{\{X^x_\infty=N\}}\right)\right]\\[2mm]
		&= \sum_{I\subset \{x_1,\ldots,x_n\} \ : \ |I|=\ell}\  \sum_{J\subset \{x_1,\ldots,x_n\} \setminus I} (-1)^{|J|} \hat\P_{I\cup J} (\xi_\infty (N)= |I\cup J|).
	\end{align*}
\end{adjustwidth}
	Exchanging the order of the above summations we get 
	\begin{align*}
		\hat \P_{\{x_1,\ldots, x_n\}}(\xi_\infty(N)=\ell)
		& = \sum _{ K\subset \{x_1,\ldots,x_n\} \ : \ |K|\ge \ell }\hat\P_{K} (\xi_\infty (N)= |K|) \left( \sum_{I \subset K\ : \ |I|=\ell}(-1)^{|K|-|I|}\right)\\[2mm]
		&=\sum _{ K\subset \{x_1,\ldots,x_n\} \ : \ |K|\ge \ell } (-1)^{|K|-\ell} {{|K|}\choose {\ell}} \hat\P_{K} (\xi_\infty (N)= |K|)
	\end{align*}
	concluding the proof.
\end{proof}

\begin{proof}[Proof of Proposition \ref{proposition: n point correlation SSEP}]
	
	For what concerns the centered $n$ point correlations, equations \eqref{eq: centered n points correlations} and \eqref{eq: psi}  follows by the combination of Theorem \ref{theorem: absorbtion probability SSEP conductances} above (which is proved in Section \ref{section: 2 point correlations} and Section \ref{section: Ninja method} below with two different approaches) with \cite[Theorem 5.1]{FRS_ort_dual}.
	
	For what concerns the non-centered $n$ point correlations, recall from \eqref{eq: duality computation}  that

	\begin{align*}
		\E_{\mus}\left[\prod_{i=1}^n\eta(x_i)\right]=\sum_{\ell=0}^n \rho_R^\ell \rho_L^{n-\ell}\hat\P_{\{x_1,\ldots,x_n\}}[\xi_\infty(N)=\ell].
	\end{align*}
	Employing Proposition \ref{proposition: general absorption prob} we obtain 
	
	\begin{align*}
		&	\E_{\mus}\left[\prod_{i=1}^n\eta(x_i)\right]=\sum_{\ell=0}^n \rho_R^\ell \rho_L^{n-\ell} \sum_{k=\ell}^n (-1)^{k-\ell} {{k} \choose{\ell}}\sum_{1\le i_1<\ldots<i_k\le n}\hat \P_{\{x_{i_1},\ldots, x_{i_k}\}}(\xi_\infty(N)=k) 
	\end{align*}
	and changing the order of summations we have 
	\begin{align*}
	\E_{\mus}\left[\prod_{i=1}^n\eta(x_i)\right] & =        \sum_{k=0}^n\ \sum_{1\le i_1<\ldots<i_k\le n}\hat \P_{\{x_{i_1},\ldots, x_{i_k}\}}(\xi_\infty(N)=k)    \sum_{\ell=0}^k  (-1)^{k-\ell} {{k} \choose{\ell}}\rho_R^\ell \rho_L^{n-\ell} \\
		&=        \sum_{k=0}^n\ \sum_{1\le i_1<\ldots<i_k\le n}\hat \P_{\{x_{i_1},\ldots, x_{i_k}\}}(\xi_\infty(N)=k)   \rho_L^{n-k} \sum_{\ell=0}^k   {{k} \choose{\ell}}\rho_R^\ell (-\rho_L)^{k-\ell} \\
		&=  \sum_{k=0}^n \rho_L^{n-k} (\rho_R-\rho_L)^{k}\sum_{1\le i_1<\ldots<i_k\le n}\hat \P_{\{x_{i_1},\ldots, x_{i_k}\}}(\xi_\infty(N)=k)  
	\end{align*}
	
	and employing Theorem \ref{theorem: absorbtion probability SSEP conductances} above we conclude the proof.

\end{proof}

\section{2-point correlations and the absorption probabilities via induction}\label{section: 2 point correlations}
In this section, our objective is to calculate the two-point stationary correlations for the boundary driven SSEP on a segment where all the conductances are set to one, except for the conductances that connect site $1$ with the left reservoir and site $N-1$ with the right reservoir. Thus we consider the $\alpha$, $\beta$, $\gamma$, $\delta$ model (see, e.g., \cite{gantert2020mixing, goncalves_non-equilibrium_2019}) and we allow the possibility to rescale the intensity of the conductances connected to reservoirs as done in hydrodynamic limits (see \cite{goncalves_survey}).

Because by duality, this amount to compute the absorption probabilities of two dual interacting particles, we start by computing such quantities in Subsection \ref{subsection: 2 point correlations} below. Based on the resulting expression, we then make an educated guess for the absorption probability of an arbitrary number of particles, which we subsequently prove to be true by induction in Subsection \ref{subsection: induction} below.

Additionally, in Section \ref{section: Ninja method} below, we present an alternative approach to compute the absorption probabilities. This method relies on a probabilistic coupling technique that eliminates the need to make assumptions regarding the specific form of these probabilities.

% and then, from the expression that we obtain we make a guess for the absorption probability of an arbitrary amount of particles that we prove by induction in Subsection \ref{subsection: induction} below.

%In Section \ref{section: Ninja method} below instead, we provide an alternative approach to compute the absorption probabilities, based on a probabilistic coupling which does not require any guess on the form of such probabilities.
 \subsection{2-point correlations  via three martingales}\label{subsection: 2 point correlations}
 Consider the open SSEP on  the homogeneous segment $([N-1], E, (\omega_{x,y})_{\{x,y\}\in E})$, with $\{x,y\}\in E$ if and only if $|x-y|=1$, $\omega_{x,x+1}=1$ for each $x\in\{1,\ldots,N-2\}\in E$, and set $\omega_L,\omega_R>0$. 
 %\begin{align*}
 	%\begin{split}
 	%	\rho_L =\frac{\alpha}{\alpha + \gamma} >0\\
 	%	\rho_L =\frac{\delta}{\delta+ \beta} >0
 %	\end{split}
 %	\begin{split}
 	%	\omega_L =&\frac{1}{\alpha + \gamma}>0\\
 	%	\omega_R =&\frac{1}{\delta + \beta}>0.
 	%\end{split}
% \end{align*}

 We  prove the following expression for the two point centered correlation previously derived  in \cite{sphon88} and \cite{derrida_entropy_2002} via other methods.
\begin{proposition}\label{proposition: 2 point correlation SSEP}
 For $x<y$
 
\begin{adjustwidth}{-0.40in}{-0.40in}
\begin{align*}
{\rm{Cor}}(x,y) & :=\hat \E_{\mus}[\eta(x)\eta(y)]-\hat \E_{\mus}[\eta(x)]\hat \E_{\mus}[\eta(y)]\\[2mm]
& \hspace{2mm} =-(\rho_R-\rho_L)^2  \frac{\left(\frac{1}{\omega_L}+x-1\right)\left(\frac{1}{\omega_R}+N-1-y\right)}{\left(\frac{1}{\omega_L}+\frac{1}{\omega_R}+N-2\right)^2\left(\frac{1}{\omega_L}+\frac{1}{\omega_R}+N-3\right)} .
\end{align*}
\end{adjustwidth}
\end{proposition}

We achieve the above result by computing the absorption probabilities of two dual particles evolving on the extended segment $[N-1]\cup\{0,N\}$, where $\{0,N\}$ are the two absorbing sites. In particular, in order to compute  asymptotic absorption probabilities,  it is enough to consider the skeleton chain of the dual process $(\xi_t)_{t\ge 0}$ starting from $\xi_0=\{x,y\}$. Denote by $(X^x_n, X^y_n)$ the position of two dual particles starting from $x<y$, where swapping is not allowed, after $n$ steps (i.e. $n$ is the number of jumps). Notice that $X^x_n=X^y_n$ if and only if $X^x_n=X^y_n\in\{0,N\}$, i.e. both particles are eventually absorbed in the same site. 

Recall that in a segment weighted by symmetric conductances $\omega_{x,x+1}$, harmonic functions at $z$ are given by 
$$\sum_{i=0}^{z-1}\frac{1}{\omega_{i,i+1}}+\text{constant}.$$
Thus, in our setting we set 
\begin{equation}\label{eq: harmonic function}
\begin{cases}
h(0):=0\\
h(x):=\frac{1}{\omega_L}+x-1, \quad x\in\{1,\ldots,N-1\}\\
h(N):= \frac{1}{\omega_L}+\frac{1}{\omega_R}+N-2
\end{cases}
\end{equation}
and the absorption probability on the point $N$ of a particle starting from $x\in[N]_0$ is then given by 
$$\hat \P_{\{x\}}[\xi_\infty(N)=1]=\frac{h(x)}{h(N)}.$$

 In order to compute  the absorption probabilities we take advantage of three martingales with respect to the natural filtration generated by $\{X^x_n, X^y_n\}_{n\in \N}$ which are given in  Lemma  \ref{proposition: martingale SSEP} below.
 
\begin{lemma}\label{proposition: martingale SSEP} Let $x<y$ and define the following process
	\begin{align*}
		& D_0=0 \\[2mm]
        & D_n:=\sum_{k=0}^{n-1}\1_{\{X^y_k-X^x_k=1\}}\1_{ \{X^y_k,X^x_k\notin\{0,N\}\}} \left(  \1_{ \{X^y_k,X^x_k\notin\{1,N-1\}\}} +\frac{2}{\omega_L+1}\1_{ \{X^x_k=1\}}   +\frac{2}{\omega_R+1}\1_{ \{X^y_k=N-1\}}           \right).
	\end{align*}
Then the processes
 	\begin{align}\label{eq: martingale sum SSEP}
 	(h(X^x_n) + h(X^y_n))_{n\ge 0},
 	\end{align}
 	\begin{align}
(h(X^y_n) -h(X^x_n)-D_n)_{n\ge 0}
 	\end{align}
 	and
 	\begin{align}
 	(h(X^x_n)h(X^y_n)+ \frac{1}{2}D_n)_{n\ge 0}
 	\end{align}
 	are martingales with respect to the filtration $\mathcal F_n:=\sigma \{  X^x_k, X^y_k, k\in \{0,\ldots, n\}\}$. 
 \end{lemma}
\begin{proof}
We provide only the proof of the fact that $(h(X^x_n)h(X^y_n)+ \frac{1}{2}D_n)_{n\ge 0}$ is a martingale since for the other two processes the conclusion follows by similar arguments. 

Denote by $\hat \E_{(x,y)}$ the expectation of the skeleton chain  $\{X^x_n, X^y_n\}_{n\in \N}$. We need to show that 

\begin{align}\label{eq: martingale property for product}
	\hat \E_{(x,y)}\left[ h(X^x_n)h(X^y_n)+ \frac{1}{2}D_n | \mathcal F_{n-1}\right]= h(X^x_{n-1})h(X^y_{n-1})+ \frac{1}{2}D_{n-1}.
	\end{align}

Let us now write
\begin{align}\label{eq: 1 written as several events for product martingale}
	\nonumber1 = (1 & -\1_{ \{X^y_{n-1},X^x_{n-1}\notin\{0,N\}\}} ) + \1_{ \{X^y_{n-1},X^x_{n-1}\notin\{0,N\}\}} \left(\1_{\{X^y_{n-1}-X^x_{n-1}\ne1\}}\right. \\[2mm]
    &\left.+ \1_{\{X^y_{n-1}-X^x_{n-1}=1\}} \left( \1_{ \{X^y_{n-1},X^x_{n-1}\notin\{1,N-1\}\}}\right.\right. \left.\left. +\frac{1}{\omega_L+1}\1_{ \{X^x_{n-1}=1\}}   +\frac{1}{\omega_R+1}\1_{ \{X^y_{n-1}=N-1\}}  \right)\right).
	\end{align}
    
% \begin{align}\label{eq: 1 written as several events for product martingale}
% 	\nonumber1&= (1-\1_{ \{X^y_{n-1},X^x_{n-1}\notin\{0,N\}\}} )\\\nonumber &+ \1_{ \{X^y_{n-1},X^x_{n-1}\notin\{0,N\}\}} \left(\1_{\{X^y_{n-1}-X^x_{n-1}\ne1\}}\right.\\\nonumber &\left.+ \1_{\{X^y_{n-1}-X^x_{n-1}=1\}} \left( \1_{ \{X^y_{n-1},X^x_{n-1}\notin\{1,N-1\}\}}\right.\right.\\&\left.\left. +\frac{1}{\omega_L+1}\1_{ \{X^x_{n-1}=1\}}   +\frac{1}{\omega_R+1}\1_{ \{X^y_{n-1}=N-1\}}  \right)\right).
% 	\end{align}
First notice that if at time $n-1$ both particles are absorbed, then no extra jumps occur. If only one particle is absorbed at time $n-1$  then $D_{n-1}=D_n$, the particle not absorbed is an independent random walk and because $h$ is harmonic, we conclude  
\begin{align*}
\hat \E_{(x,y)} \Big[ (1-\1_{ \{X^y_{n-1},X^x_{n-1}\notin\{0,N\}\}} ) & (h(X^x_n)h(X^y_n)+ \frac{1}{2}D_n) | \mathcal F_{n-1} \Big] \\[2mm]
& =(1-\1_{ \{X^y_{n-1},X^x_{n-1}\notin\{0,N\}\}} ) (h(X^x_{n-1})h(X^y_{n-1})+ \frac{1}{2}D_{n-1}).
\end{align*}
Consider now 
$$\hat \E_{(x,y)}\left[ \1_{ \{X^y_{n-1},X^x_{n-1}\notin\{0,N\}\}} \1_{\{X^y_{n-1}-X^x_{n-1}\ne1\}}(h(X^x_n)h(X^y_n)+ \frac{1}{2}D_n) | \mathcal F_{n-1}\right].$$
On the event $\1_{\{X^y_{n-1}-X^x_{n-1}\ne1\}}$ we have $D_n=D_{n-1}$ and since particles are not nearest neighbor, they behave as independent random walks an thus 
\begin{multline}
\hat\E_{(x,y)}\left[ \1_{ \{X^y_{n-1},X^x_{n-1}\notin\{0,N\}\}} \1_{\{X^y_{n-1}-X^x_{n-1}=1\}}(h(X^x_n)h(X^y_n)+\frac{1}{2} D_n) | \mathcal F_{n-1}\right]\\[2mm] = \1_{ \{X^y_{n-1},X^x_{n-1}\notin\{0,N\}\}} \1_{\{X^y_{n-1}-X^x_{n-1}\ne1\}}(h(X^x_{n-1})h(X^y_{n-1})+ \frac{1}{2}D_{n-1}).
\end{multline}
Denote by $$A:=\1_{ \{X^y_{n-1},X^x_{n-1}\notin\{0,N\}\}} \1_{\{X^y_{n-1}-X^x_{n-1}=1\}}\1_{ \{X^y_{n-1},X^x_{n-1}\notin\{1,N-1\}\}}$$
and notice that in this case 
$$\hat \E_{(x,y)}\left[ AD_n | \mathcal F_{n-1}\right]=A(D_{n-1}+1).$$
Moreover
\begin{align*}
	&\hat \E_{(x,y)}\left[ Ah(X^x_n)h(X^y_n) | \mathcal F_{n-1}\right]=A\hat \E_{(X^x_{n-1},X^y_{n-1})}\left[ h(X^x_n)h(X^y_n)\right]\\
	&=A\left(  \frac{1}{2} h(X^x_{n-1}-1)h(X^y_{n-1}) + \frac{1}{2} h(X^x_{n-1})h(X^y_{n-1}+1)  \right)\\
	&= Ah(X^x_{n-1})h(X^y_{n-1})   -\frac{1}{2} A( h(X^y_{n-1})-h(X^x_{n-1}  )) = Ah(X^x_{n-1})h(X^y_{n-1})   - \frac{1}{2} A
\end{align*}
from which we obtain 

$$\hat \E_{(x,y)}\left[ A(h(X^x_n)h(X^y_n)+ \frac{1}{2}D_n)| \mathcal F_{n-1}\right]= A(h(X^x_{n-1})h(X^y_{n-1})+ \frac{1}{2}D_{n-1}).$$

Denote by $$B:=\1_{ \{X^y_{n-1},X^x_{n-1}\notin\{0,N\}\}} \1_{\{X^y_{n-1}-X^x_{n-1}=1\}}\1_{ \{X^x_{n-1}=1\}}$$
and notice that in this case 
$$\hat \E_{(x,y)}\left[ BD_n | \mathcal F_{n-1}\right]=B(D_{n-1}+\frac{2}{\omega_L+1}).$$
Moreover
\begin{align*}
	&\hat \E_{(x,y)}\left[ Bh(X^x_n)h(X^y_n) | \mathcal F_{n-1}\right]=B\hat \E_{(X^x_{n-1},X^y_{n-1})}\left[ h(X^x_n)h(X^y_n)\right]\\
	&=B\left(  \frac{\omega_L}{1+\omega_L} h(X^x_{n-1}-1)h(X^y_{n-1}) + \frac{1}{1+\omega_L} h(X^x_{n-1})h(X^y_{n-1}+1)  \right)\\
	&= 0+ B\frac{1}{1+\omega_L} \left( h(X^x_{n-1})h(X^y_{n-1}) +h(X^x_{n-1})\right) \\ 
	&= Bh(X^x_{n-1})h(X^y_{n-1})   -\frac{B}{1+\omega_L} ( h(X^y_{n-1})-h(X^x_{n-1}  )) = Bh(X^x_{n-1})h(X^y_{n-1})   - \frac{B}{1+\omega_L}
\end{align*}
from which we obtain 

$$\hat \E_{(x,y)}\left[ B(h(X^x_n)h(X^y_n)+ \frac{1}{2}D_n)| \mathcal F_{n-1}\right]= B(h(X^x_{n-1})h(X^y_{n-1})+ \frac{1}{2}D_{n-1}).$$

Denote by $$C:=\1_{ \{X^y_{n-1},X^x_{n-1}\notin\{0,N\}\}} \1_{\{X^y_{n-1}-X^x_{n-1}=1\}}\1_{ \{X^y_{n-1}=N-1\}}$$
and notice that in this case 
$$\hat \E_{(x,y)}\left[ CD_n | \mathcal F_{n-1}\right]=C(D_{n-1}+\frac{2}{\omega_R+1}).$$
Moreover
\begin{align*}
	&\hat \E_{(x,y)}\left[ Ch(X^x_n)h(X^y_n) | \mathcal F_{n-1}\right]=C\hat \E_{(X^x_{n-1},X^y_{n-1})}\left[ h(X^x_n)h(X^y_n)\right]\\
	&=C\left(  \frac{1}{1+\omega_R} h(X^x_{n-1}-1)h(X^y_{n-1}) + \frac{\omega_R}{1+\omega_R} h(X^x_{n-1})h(X^y_{n-1}+1)  \right)\\
	&= \frac{C}{1+\omega_R}\left[ \left( h(X^x_{n-1})h(X^y_{n-1}) -h(X^y_{n-1})\right) + \omega_R \left( h(X^x_{n-1})h(X^y_{n-1}) +\frac{1}{\omega_R}h(X^x_{n-1})\right)\right] \\ 
	&= Ch(X^x_{n-1})h(X^y_{n-1})   -\frac{C}{1+\omega_R} ( h(X^y_{n-1})-h(X^x_{n-1}  )) = Ch(X^x_{n-1})h(X^y_{n-1})   - \frac{C}{1+\omega_R}
\end{align*}
from which we obtain 

$$\hat \E_{(x,y)}\left[ C(h(X^x_n)h(X^y_n)+ \frac{1}{2}D_n)| \mathcal F_{n-1}\right]= C(h(X^x_{n-1})h(X^y_{n-1})+ \frac{1}{2}D_{n-1}).$$
Putting everything together and recalling \eqref{eq: 1 written as several events for product martingale} we obtain \eqref{eq: martingale property for product}, and thus that the process $(h(X^x_n)h(X^y_n)+ \frac{1}{2}D_n)_{n\ge 0}$ is a martingale.

\end{proof}

 We can thus prove Proposition \ref{proposition: 2 point correlation SSEP}.
 \begin{proof}[Proof of Proposition \ref{proposition: 2 point correlation SSEP}]
 By employing the martingales of Lemma \ref{proposition: martingale SSEP} above and  sending $n$ to infinity we obtain (by Doob's optional sampling theorem), for $x < y$,
 \begin{align}
 \begin{cases}
 h(x)+h(y)=2h(N)\hat \P_{\{x,y\}}[\xi_\infty(N)=2] + h(N)\hat \P_{\{x,y\}}[\xi_\infty(N)=1]\\[2mm]
 h(y)-h(x)=h(N)\hat \P_{\{x,y\}}[\xi_\infty(N)=1]-\hat \E_{(x,y)}[D_\infty]\\[2mm]
 h(x)h(y)=h(N)^2\hat \P_{\{x,y\}}[\xi_\infty(N)=1]+ \frac{1}{2}\hat \E_{(x,y)}[D_\infty]
 \end{cases}
 \end{align}
 and thus
 \begin{align}\label{eq: PxyNN SSEP}
 \hat\P_{\{x,y\}}[\xi_\infty(N)=2] =\frac{h(x)(h(y)-1)}{h(N)(h(N)-1)},
 \end{align}

 \begin{align}\label{eq: Pxy0N SSEP}
  \hat\P_{\{x,y\}}[\xi_\infty(N)=1] =  \frac{h(x)+h(y)}{h(N)}-2\frac{h(x)(h(y)-1)}{h(N)(h(N)-1)}
 \end{align}
 and 
\begin{equation}\label{eq: Pxy00 SSEP}
\hat\P_{\{x,y\}}[\xi_\infty(N)=0] =1-  \hat\P_{\{x,y\}}[\xi_\infty(N)=2]-\hat\P_{\{x,y\}}[\xi_\infty(N)=1] =1-\frac{h(x)+h(y)}{h(N)} +\frac{h(x)(h(y)-1)}{h(N)(h(N)-1)}.
\end{equation}
Writing
$$\hat \E_{\mus}[\eta(x)\eta(y)]=\rho_L^2\hat\P_{\{x,y\}}[\xi_\infty(N)=0] + \rho_L\rho_R  \hat\P_{\{x,y\}}[\xi_\infty(N)=1]  + \rho_R^2 \hat\P_{\{x,y\}}[\xi_\infty(N)=2]$$
and 
\begin{adjustwidth}{-0.43in}{-0.43in}
\begin{align*}  \hat \E_{\mus}[\eta(x)]\hat \E_{\mus}[\eta(y)] &=\rho_R^2\hat \P_{\{x\}}[\xi_\infty(N)=1]\hat \P_{\{y\}}[\xi_\infty(N)=1] +\rho_R\rho_L(\hat \P_{\{x\}}[\xi_\infty(N)=0]\hat \P_{\{y\}}[\xi_\infty(N)=1] \\
& \hspace{4mm}+\hat \P_{\{x\}}[\xi_\infty(N)=1]\hat \P_{\{y\}}[\xi_\infty(N)=0]) +\rho_L^2\hat \P_{\{x\}}[\xi_\infty(N)=0]\hat \P_{\{y\}}[\xi_\infty(N)=0]
\end{align*}
\end{adjustwidth}
 and recalling  that $\hat \P_{\{x\}}[\xi_\infty(N)=1]=\frac{h(x)}{h(N)}$, we finally obtain, using \eqref{eq: PxyNN SSEP}, \eqref{eq: Pxy0N SSEP} and \eqref{eq: Pxy00 SSEP}, the desired identity.
\end{proof}
 
\subsection{Absorption probabilities via induction}\label{subsection: induction}
In the proof above, we showed that for $x<y$,
$$ \hat\P_{\{x,y\}}[\xi_\infty(N)=2] =\frac{h(x)(h(y)-1)}{h(N)(h(N)-1)} = \frac{\frac{1}{\omega_L}+x-1}{\frac{1}{\omega_L}+\frac{1}{\omega_R}+ (N-1)-1}\frac{\frac{1}{\omega_L}+y-2}{\frac{1}{\omega_L}+\frac{1}{\omega_R}+ (N-1)-2}. $$ 
From the formula above we can make a guess for the absorption probability that starting with $k$ particles all of them are absorbed at $N$ and then prove that our guess is correct by induction.
\begin{lemma}\label{lemma: abs prob conductances at boundary}
Recall the function $h$ given in \eqref{eq: harmonic function}. Then, given $1\le x_1<\ldots <x_k\le N-1$,
 \begin{align}
	\hat\P_{\{x_1,\ldots,x_k\}}[\xi_\infty(N)=k]=\prod_{i=1}^k \frac{h(x_i)-(i-1)}{h(N)-(i-1)}.
\end{align}
\end{lemma}
%from which one can make the guess that, given $x_1<\ldots<x_k$, then 
%We can prove that the above formula is correct by induction.
\begin{proof}
First notice that the formula matches the boundary conditions of the absorption probability.
Second, assume that the formula holds for $k-1$ and let us show that 
$$\hat{\mathcal L} \left(\prod_{i=1}^k \frac{h(x_i)-(i-1)}{h(N)-(i-1)}\right) =0.$$

Define 
$$g_k(x_1,\ldots,x_k):=\prod_{i=1}^k \frac{h(x_i)-(i-1)}{h(N)-(i-1)}.$$
If $x_{k}>x_{k-1}+1$ then we can apply the dual generator in the following way 
\begin{adjustwidth}{-0.45in}{-0.45in}
\begin{align*}
	\hat{\mathcal L} g_k(x_1,\ldots,x_k) =  \left(\hat{\mathcal L} g_{k-1}(x_1,\ldots,x_{k-1}) \right) \frac{h(x_k)-(k-1)}{h(N)-(k-1)}  + g_{k-1}(x_1,\ldots,x_{k-1}) \hat{\mathcal L} \left(\frac{h(x_k)-(k-1)}{h(N)-(k-1)} \right)   =0
\end{align*}
\end{adjustwidth}
where the last equality follows by the induction assumption and the fact that $h$ is harmonic for $\hat{\mathcal L}$.
If $x_{k}=x_{k-1}+1$ and $x_{k+1}\ne N-1$ then we can apply the generator in the following way 
\begin{adjustwidth}{-0.50in}{-0.50in}
\begin{align*}
	\hat{\mathcal L} \left(\prod_{i=1}^k \frac{h(x_i)-(i-1)}{h(N)-(i-1)}\right) & =  \left(\hat{\mathcal L} g_{k-1}(x_1,\ldots,x_{k-1}) \right) \frac{h(x_k)-(k-1)}{h(N)-(k-1)} \\[2mm]
    &-g_{k-2}(x_1,\ldots,x_{k-2}) \left(\frac{h(x_{k-1}+1)-(k-2)}{h(N)-(k-2)}-\frac{h(x_{k-1})-(k-2)}{h(N)-(k-2)}\right)\frac{h(x_k)-(k-1)}{h(N)-(k-1)} \\[2mm]
    &+g_{k-1}(x_1,\ldots,x_{k-1}) \left(\frac{h(x_{k}+1)-(k-1)}{h(N)-(k-1)}-\frac{h(x_{k})-(k-1)}{h(N)-(k-1)}\right) =0
\end{align*}
\end{adjustwidth}
where the first term on the right hand side is zero by induction and the second term, which appears because in the first term of the right hand side we applied the generator as if the $k+1$ particle was not present in the system, cancel the third term. %Similarly, 
If $x_{k}=x_{k-1}+1$ and $x_{k+1}= N-1$, the difference with respect to the computation above is that the third term on the right hand side is multiplied by $\omega_R$ which however cancel with $h(x_k+1)-h(x_k)=\frac{1}{\omega_R}$ and thus again $\hat{\mathcal L} \left(\prod_{i=1}^k \frac{h(x_i)-(i-1)}{h(N)-(i-1)}\right)=0$ concluding the proof.
\end{proof}

\section{n-point correlations via the \textit{Ninja} particle method}\label{section: Ninja method}
 The matrix ansatz method of \cite{derrida_exact_1993-1} allows to derive a recursive relation for the correlations of the boundary driven SSEP (see \cite[(A.7)]{derrida_entropy_2007}). As pointed out in \cite[Section 7.1]{carinci_consistency}, from such relation satisfied by the correlations, an analogous relation for the absorption probabilities in the dual system follows. However, so far a probabilistic approach alternative to the matrix ansatz method has not been proposed.

 In this section we  provide a probabilistic route to compute the absorption probabilities of the dual process by building a coupling between two dual boundary driven SSEP, one evolving on $[N]_0$ with $k+1$ particles and the other one on $[N-1]_0$ with $k$ particles. From such coupling we deduce a recursive relation for the absorption probabilities in the two systems. In the following we consider the homogeneous segment, i.e. $\{x,y\}\in E$ if and only if $|x-y|=1$, $\omega_{x,x+1}=1$ for each $x\in\{1,\ldots,N-2\}\in E$ and we put $\omega_L=\omega_R=1$ as well. We denote by  $(\xi^{[M]}_t)_{t\ge 0}$  the dual boundary driven SSEP  on the homogeneous segment $[M]_0$ where $\{0,M\}$ are absorbing sites and by $\hat \P^{[M]_0}_{\{x_1,\ldots,x_k\}}$ the law of $(\xi^{[M]}_t)_{t\ge 0}$ starting from $\sum_{i=1}^k \delta_{x_i}$. Moreover we denote by $\hat {\mathcal L}^{[M]_0}$ the generator of  $(\xi^{[M]}_t)_{t\ge 0}$.

The main result of the section is the following theorem.
 
 \begin{theorem}\label{theorem: recursion absorption probabilities via Ninja}
For all  $0\le x_1<\ldots<x_{k+1}\le N$ the following relation holds 
 %between the absorption probabilities of $(\xi^{[N]}_t)_{t\ge 0}$  and $(\xi^{[N-1]}_t)_{t\ge 0}$:
 \begin{align}
\hat \P^{[N]_0}_{\{x_1,\ldots,x_{k+1}\}}(\xi^{[N]_0}_\infty(N)=k+1)=\left(\frac{x_{k+1}-k}{N}\right) \ \hat \P^{[N-1]_0}_{\{x_1,\ldots,x_k\}}(\xi^{[N-1]_0}_\infty(N-1)=k).
 \end{align}
% where  \begin{align}
% \hat \P^{[N]_0}_{x_1,\ldots,x_k}=\hat \P_{\{x_1,\ldots,x_k\}}(\xi^{[N]_0}_\infty(N)=k).
% \end{align}
 \end{theorem}
Recalling that $\hat \P^{[N]_0}_{\{x\}} (\xi^{[N]_0}_\infty(N)=1)=\frac{x}{N}$, as a direct consequence of the above theorem, we obtain the explicit absorption probabilities.
 \begin{corollary}
Let $0< x_1<\ldots<x_{k}< N$, then
\begin{align}\label{eq: n points absorption prob}
\hat \P^{[N]_0}_{\{x_1,\ldots,x_{k}\}} (\xi^{[N]_0}_\infty(N)=k)=\prod_{i=1}^k \frac{x_i-(i-1)}{N-(i-1)}.
\end{align} 
 \end{corollary}
 In the next section we present the coupling we mentioned above. This technique relies on the introduction of a special particle, denoted by \textit{Ninja}, which has some special features reminiscent of the behavior of second class particles (see, e.g., \cite[pag. 218]{liggett_stochastic_1999}) and will be responsible for the coupling of the two processes on different graphs. 
 \subsection{The \textit{Ninja}-process: a game of labels}
 
 Consider the homogeneous segment  $[N]_0=[N-1]\cup\{0,N\}$ where $\{0,N\}$ are absorbing sites.
 
 We place $k+1$ particles initially at distinct positions $x_1,\ldots,x_{k+1}\in[N]_0$ with
 $$1\le x_1<\ldots<x_{k}\le N-1$$
 and 
 $$x_{k+1}\notin\{x_1,\ldots,x_k\}$$
 and we label by $(i)$ the particle starting from $x_i$ for all $i\in\{1,\ldots,k\}$ and  by $\textit{Ninja}$, the particle starting at $x_{k+1}$ (see Figure \ref{fig:ninjaparticle}).  We do not allow the following two cases:
 $$\exists i\in\{1,\ldots,k\} \ s.t.\ x_i=N-1\  \text{and }x_{k+1}=N$$
 and 
 $$\exists i\in\{1,\ldots,k\} \ s.t.\ x_i=1\  \text{and }x_{k+1}=0  .$$
 For our purposes we will be interested in the case $x_{k+1}>x_i$ for all $i\in\{1,\ldots,k\}$, but the process that we are going to define does not require such condition.
 We are now going to describe in words the dynamics of the Markov process 
 $$(X^{(1)}_t,\ldots, X^{(k)}_t, \textit{Ninja}_t)_{t\ge 0}$$
 that we call \textit{Ninja}-process with law denoted by $\P^{\textit{Ninja}}_{x_1,\ldots,x_{k+1}}$.

 \begin{figure}
 	\centering
 	\includegraphics[width=0.7\linewidth]{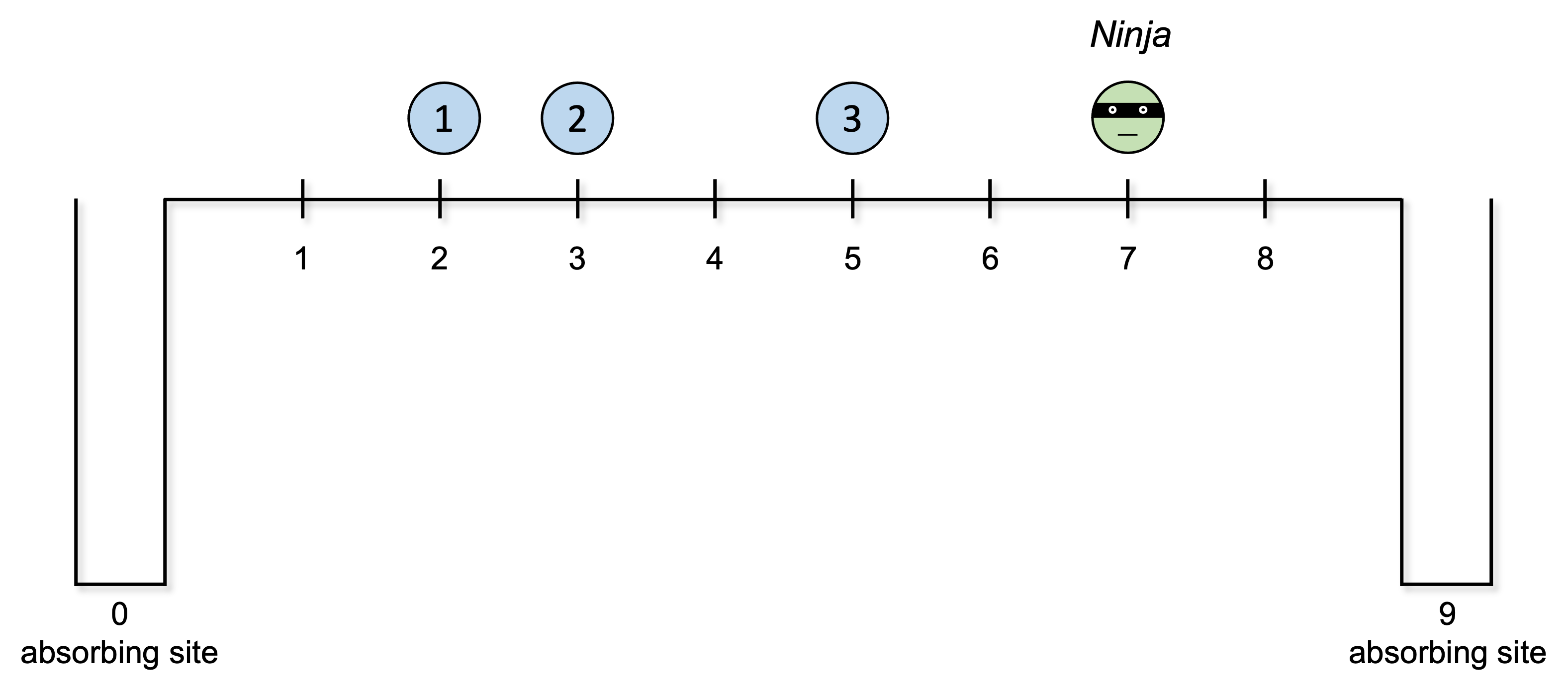}
    \caption{The Ninja-process}
 	\label{fig:ninjaparticle}
 \end{figure}

The particles will move in such a way that the configuration process obtained by the labelled particles evolve exactly as the usual dual of the boundary driven SSEP on $[N]_0$ but we play with the labels of the particles and more precisely with the \textit{Ninja} label in order to perform the coupling. More precisely:

%Each site in $[N-1]$ can be occupied by maximum one particle. Each particle has an exponential random clock with parameter $2$, when a clock rings the particle associated to that clock may perform a jump with equal probability to its left and right site, depending on the occupancy of its nearest neighboring sites. More precisely:
 
 \begin{itemize}
 	\item[\textbf{Case 1}] (The \textit{Ninja} is alone) When no particle is in a location nearest neighbor to the \textit{Ninja} and the $\textit{Ninja}\notin\{0,N\}$ each particle, including the \textit{Ninja}, behaves as  simple symmetric random walks jumping at rate one, subject to the exclusion rule and with $\{0,N\}$ acting as absorbing sites.  In this case, the \textit{Ninja} does not have nearest neighbor particles and it  jumps at rate one either to its left or to its right. % If a particle with $i\in\{1,\ldots,k\}$ is at $0$ or $N$ will stay there forever, but this is not the case for the Ninja (see \textbf{Case 3}).
 	\item[\textbf{Case 2}]  (The \textit{Ninja} interacts) Consider now the case in which  the \textit{Ninja} particle is not at $\{0,N\}$ and the $i$-th particle, with $i\in\{1,\ldots, k\}$, is nearest neighbor of the \textit{Ninja} and not in $\{0,N\}$. Then all the other particles that are not nearest neighbor of  the \textit{Ninja} behave as simple symmetric random walks jumping at rate one, subject to the exclusion rule and with $\{0,N\}$ acting as absorbing sites. 
 	On the other hand the couple $(X^{(i)}, \textit{Ninja})=(x_i, x_{k+1})$ jumps at rate one to 
 	$$(x_i+2(x_{k+1}-x_i), x_i)$$
 	if the site $x_i+2(x_{k+1}-x_i)$ is empty and at rate one to 
 	$$(x_i-(x_{k+1}-x_i), x_{k+1})$$
 	 if $x_i-(x_{k+1}-x_i)$ is empty 
 	 (see Figures \ref{fig:ninjaparticle2} and \ref{fig:ninjaparticle_block}). Similarly if there exists another index $j\ne i$ such that $|X^{(j)}-\textit{Ninja}|=|x_j-x_{k+1}|=1$ and $x_j\notin\{0,N\}$, then the couple $(X^{(j)}, \textit{Ninja})$ behaves analogously to the couple $(X^{(i)}, \textit{Ninja})$.

 \begin{figure}
 		\centering
 		\includegraphics[width=0.6\linewidth]{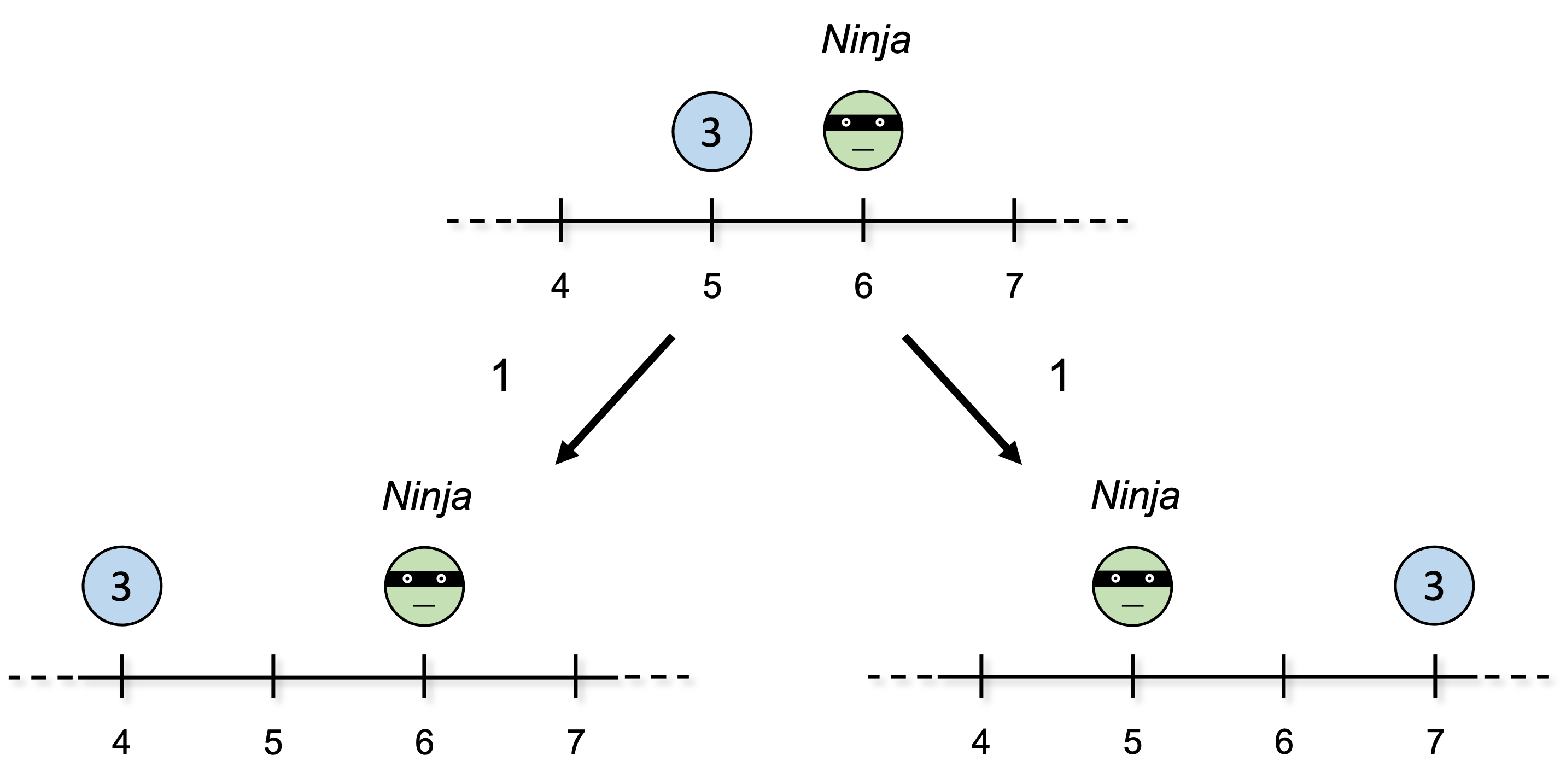}
 		\caption{The Ninja interacts}
 		\label{fig:ninjaparticle2}
 	\end{figure}
 
   \begin{figure}
 	\centering
 	\includegraphics[width=0.8\linewidth]{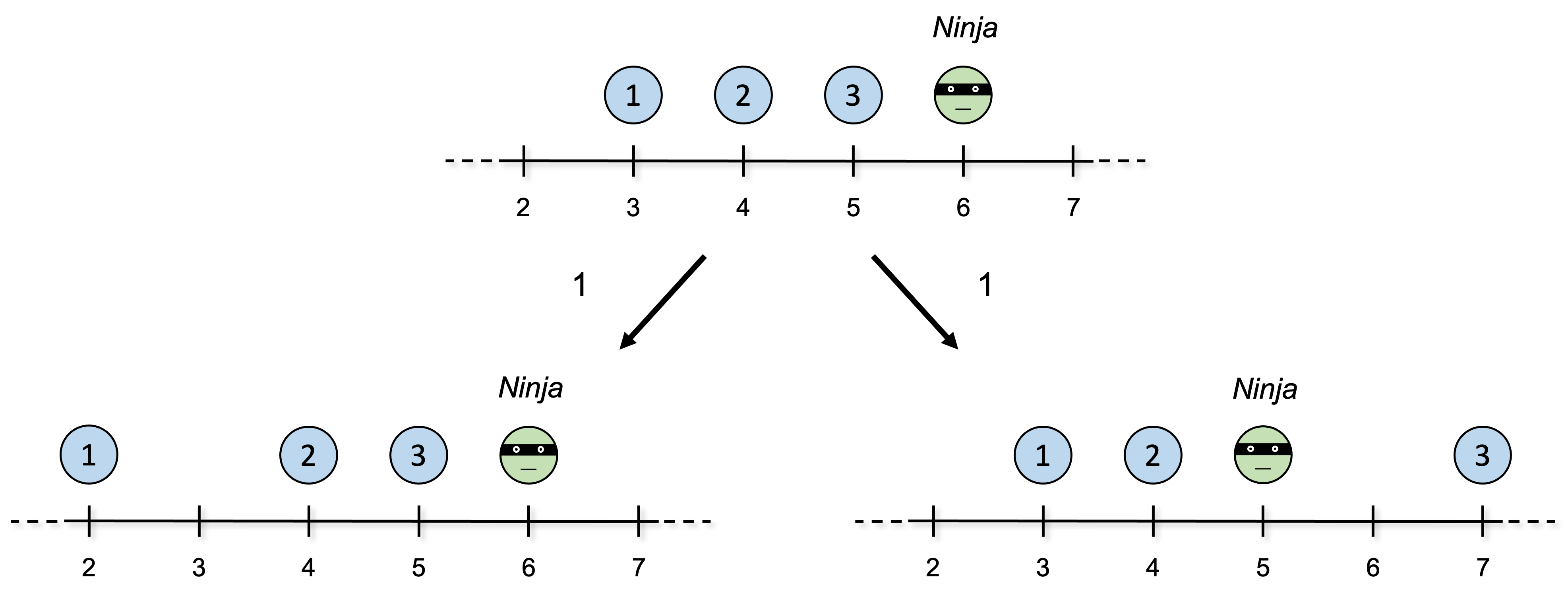}
 	\caption{The Ninja interacts}
 	\label{fig:ninjaparticle_block}
 \end{figure}
 
 	\item[\textbf{Case 3}]  (The \textit{Ninja} returns)  %In the Ninja particle system, it is not allowed the situation in which at some time  the \textit{Ninja}  is at site $N$ and one of the other particles is at $N-1$. If the Ninja is at $N$ and its clock rings, it will not move. However, the Ninja can escape from the absorbing site $N$ when one of the other particle is at $N-2$. Thus if $Ninja_{t^-}=N$ and $X^{x_i}_{t^-}=N-2$ and the clock of the $i-th$ particle rings, then with equal probability a direction is chosen between left and right. If the left direction is chosen the $i$-th particle jumps on the site $N-3$ if empty otherwise it does not move. If the right direction is chosen, then $X^{x_i}_{t}=N$ and will stay there forever while  $Ninja_{t}=N-1$, namely the $i$-th particle jumps to the absorbing point and the $\textit{Ninja}$ escapes from the absorbing point and moves to $N-1$. 
 	%The analogous dynamics occurs at site $0$. Namely the above described movements occurs also when ones substitutes $0$ with $N$, $1$ with $N-1$ and $2$ with $N-2$ and left with right.
 	In the Ninja particle system, it is not allowed the situation in which  the \textit{Ninja}  is at site $x\in\{N, 0\}$ and one of the other particles is at $|x-1|$. If the \textit{Ninja} is at $x\in\{N, 0\}$, it  can escape from the absorbing site $x\in\{N, 0\}$ when one of the other particles tries to jump  at $|x-1|$. Indeed suppose that $X^{(i)}=|x-2|$, then all other particles behave as simple symmetric random walks jumping at rate one, subject to the exclusion rule and with $\{0,N\}$ acting as absorbing sites, while the couple $(X^{(i)}, \textit{Ninja})=(|x-2|,x)$ jumps at rate $1$ to 
 	$$\left(  |x-2|-\frac{x-|x-2|}{2} , x\right)$$
 	if $|x-2|-\frac{x-|x-2|}{2} $ is empty and at rate $1$ to 
 	
 	$$\left(x, x-\frac{x-|x-2|}{2}\right)$$
 	(see Figure \ref{Ninja at the boundary}). 
 
 	\begin{figure}
 		\centering
 		\includegraphics[width=0.8\linewidth]{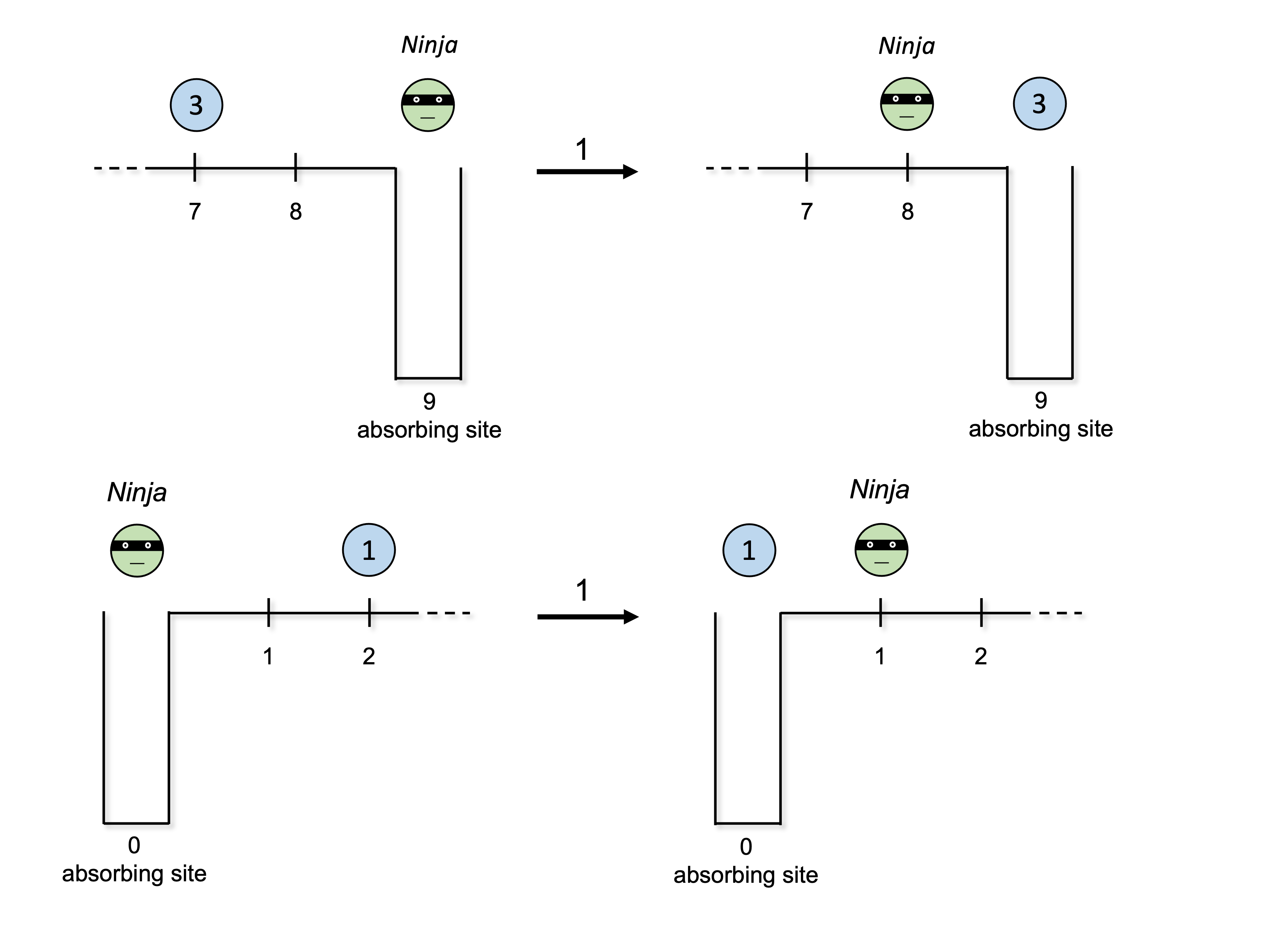}
 		\caption{Ninja at the boundary}
 		\label{Ninja at the boundary}
 	\end{figure}
 	
 \end{itemize}
 We now describe the dynamics of the \textit{Ninja}-process via its generator 
 $$\hat{\mathcal{L}}^{\textit{Ninja}}=\hat{\mathcal{L}}^{\textit{Ninja}}_1+\hat{\mathcal{L}}^{\textit{Ninja}}_2+\hat{\mathcal{L}}^{\textit{Ninja}}_3$$
 acting on functions $f:[N]^{k+1}\to \R$.
  
$\hat{\mathcal{L}}^{\textit{Ninja}}_1$ describes the dynamics in \textbf{Case 1} and it is given by 

% \begin{align*}
%& L^1 f (x_1,\ldots, x_k)= \left(\prod_{i=1}^k\one_{\{|x_i-x_{k+1}|\ne 1\}}\right)\one_{\{x_{k+1}\notin \{0,N\} \}} \\
%&\times \left(\sum_{i=1}^{k+1}\sum_{x=0}^{N-1}  \one_{\{x_i=x\}} \left(1- \sum_{j=1}^{k+1} \one_{\{x_j=x+1\}})\right) \right.\\&\left.\times (f (x_1,\ldots, x_{i-1}, x_i+1,x_{i+1},\ldots,x_{k+1})-f (x_1,\ldots, x_{k+1}))\right.\\
%&\qquad \left.  \sum_{i=1}^{k+1} \sum_{x=1}^{N} \one_{\{x_i=x\}} \left(1- \sum_{j=1}^{k+1} \one_{\{x_j=x-1\}})\right)  \right. \\&\left.\times (f (x_1,\ldots, x_{i-1}, x_i-1,x_{i+1},\ldots,x_{k+1})-f (x_1,\ldots, x_{k+1})) \right)
 %\end{align*}

 \begin{adjustwidth}{-0.50in}{-0.50in}
\begin{align*}
\hat{\mathcal{L}}^{\textit{Ninja}}_1 f (x_1,\ldots, x_{k+1}) & = \left(\prod_{i=1}^k\one_{\{|x_i-x_{k+1}|\ne 1\}}\right)\one_{\{x_{k+1}\notin \{0,N\} \}} \left(\sum_{i=1}^{k+1} \one_{\{x_i\notin\{0,N\}\}} \left(1- \sum_{j=1}^{k+1} \one_{\{x_j=x_i\pm 1\}})\right) \right.\\[2mm]
 &\qquad \times \big( f (x_1,\ldots, x_{i-1}, x_i\pm1,x_{i+1},\ldots,x_{k+1})-f (x_1,\ldots, x_{k+1}) \big) \Bigg).
 \end{align*}
\end{adjustwidth}
 
 $\hat{\mathcal{L}}^{\textit{Ninja}}_2$ describes the dynamics in \textbf{Case 2} (see Figure \ref{fig:ninjaparticle2} and \ref{fig:ninjaparticle_block}) and it is given by
  \begin{adjustwidth}{-0.50in}{-0.50in}
\begin{align*}
\hat{\mathcal{L}}^{\textit{Ninja}}_2 f (x_1,\ldots, x_{k+1}) & =\left(1-\prod_{i=1}^k\one_{\{|x_i-x_{k+1}|\ne 1\}}\right)\one_{\{x_{k+1}\notin \{0,N\} \}} \left(\sum_{i=1}^{k}  \one_{\{x_i\notin\{0,N\}\}} \left[\left(1- \sum_{j=1}^{k+1} \one_{\{x_j=x_i\pm 1\}})\right) \right.\right.\\[2mm]
 & \qquad \left. \left.  \times \big(f (x_1,\ldots, x_{i-1}, x_i\pm1,x_{i+1},\ldots,x_{k+1})-f (x_1,\ldots, x_{k+1}) \big) \right.\right.\\[2mm]
 & \quad \left. \left. + \one_{\{|x_i-x_{k+1}|=1\}} \left(1-\sum_{j=1}^{k+1}\one_{\{x_i+2(x_{k+1}-x_i)=x_j\}} \right) \right.\right.\\[2mm] 
 & \qquad \left. \times \big(f (x_1,\ldots, x_{i-1}, x_i+2(x_{k+1}-x_i),x_{i+1},\ldots,x_i)-f (x_1,\ldots, x_{k+1})\big) \Bigg] \right).
 \end{align*}
\end{adjustwidth}

$\hat{\mathcal{L}}^{\textit{Ninja}}_3$ describes the dynamics in \textbf{Case 3} (see Figure \ref{Ninja at the boundary}) and it is given by
\begin{adjustwidth}{-0.50in}{-0.50in}
\begin{align*}
\hat{\mathcal{L}}^{\textit{Ninja}}_3 f (x_1,\ldots, x_k) & =\one_{\{x_{k+1}\in \{0,N\} \}}  \times \left(\sum_{i=1}^{k}  \one_{\{x_i\notin\{0,N\}\}} \left[\one_{\{|x_i-x_{k+1}|\ne 2\}}\left(1- \sum_{j=1}^{k} \one_{\{x_j=x_i\pm 1\}})\right) \right.\right.\\[2mm]
&\qquad \left. \left.  \times \big(f (x_1,\ldots, x_{i-1}, x_i\pm1,x_{i+1},\ldots,x_{k+1})-f (x_1,\ldots, x_{k+1}) \big) \right.\right.\\[2mm]
& \quad \left. \left. +\one_{\{|x_i-x_{k+1}|= 2\}} \left(1-\sum_{j=1}^{k+1}\one_{\{x_i-(x_{k+1}-x_i)/2=x_j\}} \right)\right.\right.\\[2mm]
& \qquad \left.\left.\times \big(f (x_1,\ldots, x_{i-1}, x_i-(x_{k+1}-x_i)/2,x_{i+1},\ldots,x_{k+1})-f (x_1,\ldots, x_{k+1})\big) \right.\right.\\[2mm]
& \quad \left. +\one_{\{|x_i-x_{k+1}|= 2\}} \big(f (x_1,\ldots, x_{i-1}, x_{k+1},x_{i+1},\ldots,x_i+(x_{k+1}-x_i)/2)-f (x_1,\ldots, x_{k+1})\big)  \Bigg]\right).
 \end{align*}
\end{adjustwidth}

 \subsection{Consequences of the construction and the coupling}

 From the construction of the \textit{Ninja}-process, the two propositions below follow. 
 The first one states that if one forgets the labels in the Ninja process and looks at the configuration process associated to it, then the usual dual boundary driven process on $[N]_0$ is observed.
 
 \begin{proposition}\label{proposition: Ninja sep = dual sep}
 	Let $(X^{(1)}_t,\ldots, X^{(k)}_t, \textit{Ninja}_t)_{t\ge 0}$ be the \textit{Ninja}-process and define

 $$\mathcal N^{[N]_0}_t:= \sum_{i=1}^k \delta_{X^{(i)}_t} + \delta_{\textit{Ninja}_t}.$$
 Then, if $\xi^{[N]_0}_0= \sum_{i=1}^{k} \delta_{X^{(i)}_0} + \delta_{\textit{Ninja}_0}$, 
 $$(\mathcal  N^{[N]_0}_t)_{t\ge 0}=(\xi^{[N]_0}_t)_{t\ge 0}$$
 in distribution.
 \end{proposition}
 \begin{proof}
 	The result follows from the fact the for all $G:[N]_0^{k+1}\to \R$ permutation invariant, i.e.
 	$$G(x_1,\ldots,x_{k+1})=G\left(\sum_{i=1}^{k+1}\delta_{x_i}\right),$$
we have
 	$$\hat {\mathcal L}^{\textit{Ninja}}G(x_1,\ldots,x_{k+1})=\hat {\mathcal L}^{[N]_0}G(x_1,\ldots,x_{k+1})$$
 	for all $0\le x_1<\ldots<x_k\le N$ and  $x_{k+1}\in[N]_0$ such that, if $x_{k+1}\notin\{0,N\}$, $\sum_{i=1}^{k}\delta_{x_i}(x_{k+1})=0.$
 	\noindent 	
The above equality is easy to be checked by direct inspection. Indeed notice  that all the transitions in the \textit{Ninja}-process occur at rate one. Moreover, if $(x_1,\ldots,x_{k+1})$ is such that  $$\left(\prod_{i=1}^k\one_{\{|x_i-x_{k+1}|\ne 1\}}\right)\one_{\{x_{k+1}\notin \{0,N\} \}} =1,$$
i.e. we are in \textbf{Case 1}, 
 	then obviously 
 	$$\hat{\mathcal{L}}^{\textit{Ninja}}_1 G (x_1,\ldots, x_{k+1})=\hat{\mathcal{L}}^{[N]_0} G (x_1,\ldots, x_{k+1}).$$
If $(x_1,\ldots,x_{k+1})$ is such that 
 	
 $$	\left(1-\prod_{i=1}^k\one_{\{|x_i-x_{k+1}|\ne 1\}}\right)\one_{\{x_{k+1}\notin \{0,N\} \}} \one_{\{x_i\notin\{0,N\}\}}=1, $$
 and
 $$\one_{\{|x_i-x_{k+1}|=1\}}(1-\sum_{j=1}^{k+1}\one_{\{x_i+2(x_{k+1}-x_i)=x_j\}})=1,$$
 with, e.g. $x_{k+1}=x_i+1$, i.e. we are in \textbf{Case 2} (see Figure \ref{fig:ninjaparticle2} and \ref{fig:ninjaparticle_block}), because the transition 
 	$$(x_1,\ldots, x_{k+1})\to(x_1,\ldots, x_{i-1}, x_i+2(x_{k+1}-x_i),x_{i+1},\ldots,x_i)  $$
 	corresponds to the configuration transition 
 	
 	$$\sum_{j=1}^{k+1}\delta_{x_j}\to \sum_{j=1}^{k+1}\delta_{x_j}-\delta_{x_{k+1}}+\delta_{x_{k+1}+1}$$
 	then 
 	
 		$$\hat{\mathcal{L}}^{\textit{Ninja}}_2 G (x_1,\ldots, x_{k+1})=\hat{\mathcal{L}}^{[N]_0} G\left(\sum_{i=1}^{k+1}\delta_{x_i}\right).$$
 		
 If $(x_1,\ldots,x_{k+1})$ is such that 
 
 $$\one_{\{x_{k+1}\in \{0,N\} \}} \one_{\{x_i\notin\{0,N\}\}}\one_{\{|x_i-x_{k+1}|= 2\}},$$ 
 with, e.g., ${x_{k+1}}=N$ and thus $x_i=N-2$,
 i.e. we are in \textbf{Case 3} (see Figure \ref{Ninja at the boundary}), 
 because the transition 
 $$(x_1,\ldots, x_{k+1}) \to (x_1,\ldots, x_{i-1}, x_{k+1},x_{i+1},\ldots,x_i+(x_{k+1}-x_i)/2)   $$
 corresponds to the configuration transition 
 
 $$\sum_{j=1}^{k+1}\delta_{x_j}\to \sum_{j=1}^{k+1}\delta_{x_j}-\delta_{N-2}+\delta_{N-1}$$
 then 
 $$\hat{\mathcal{L}}^{\textit{Ninja}}_3 G (x_1,\ldots, x_{k+1})=\hat{\mathcal{L}}^{[N]_0} G\left(\sum_{i=1}^{k+1}\delta_{x_i}\right).$$
 Because for all $\{x_1,\ldots,x_{k+1}\}\in[N]_0$ such that $\sum_{i=1}^{k+1}\delta_{x_i}(x)\in\{0,1\}$ for all $x\in[N-1]$
 $$\left(\prod_{i=1}^k\one_{\{|x_i-x_{k+1}|\ne 1\}}\right)\one_{\{x_{k+1}\notin \{0,N\} \}}\\+\left(1-\prod_{i=1}^k\one_{\{|x_i-x_{k+1}|\ne 1\}}\right)\one_{\{x_{k+1}\notin \{0,N\} \}} +\one_{\{x_{k+1}\in \{0,N\} \}} =1, $$
i.e. all the admissible vectors in the $\textit{Ninja}$-process corresponds to all the admissible configurations in the dual of the boundary driven SEP, the proof is concluded.
 
 \end{proof}

 The second one, provides a map from the  \textit{Ninja}-process to the dual of the boundary driven SSEP evolving on the reduced graph $[N-1]_0$ with $\{0,N-1\}$ acting as absorbing sites, namely to $(\xi^{[N-1]_0}_t)_{t\ge 0}$. 
 \begin{proposition}\label{proposition: f map}
 For $x,y\in[N]_0$ set
 \begin{align}
 \pi(x, y):=\begin{cases}
 N-1& \text{if }\ y=x=N\\
 x-\1_{\{y<x\}}& \text{otherwise}.
 \end{cases}
 \end{align}
 Let $(X^{(1)}_t,\ldots, X^{(k)}_t, \textit{Ninja}_t)_{t\ge 0}$ be the \textit{Ninja}-process and define
 
 $$\mathcal N^{[N-1]_0}_t:= \sum_{i=1}^k \delta_{\pi(X^{(i)}_t, \textit{Ninja}_t)}.$$
 Then, if $\xi^{[N-1]_0}_0= \sum_{i=1}^{k} \delta_{\pi(X^{(i)}_0,\textit{Ninja}_0)}$, 
 $$(\mathcal  N^{[N-1]_0}_t)_{t\ge 0}=(\xi^{[N-1]_0}_t)_{t\ge 0}$$
 in distribution.
 \end{proposition}

\begin{proof}
Similarly to the proof of Proposition \ref{proposition: Ninja sep = dual sep}, the result follows from the fact the for all $G:[N]_0^{k}\to \R$ permutation invariant, i.e.
$$G(x_1,\ldots,x_{k})=G\left(\sum_{i=1}^{k}\delta_{x_i}\right),$$
we have, for all $0\le x_1<\ldots<x_k\le N$
$$\hat {\mathcal L}^{\textit{Ninja}}G(\pi(x_1,x_{k+1}),\ldots, \pi(x_k,x_{k+1}))=\hat {\mathcal L}^{[N-1]_0}G\left(\sum_{i=1}^{k}\delta_{\pi(x_i,x_{k+1})}\right),$$
for all $x_{k+1}\in[N]_0$ such that, if $x_{k+1}\notin\{0,N\}$, $\sum_{i=1}^{k}\delta_{x_i}(x_{k+1})=0.$

\noindent
We do not provide all the details and we just point out that 
\begin{itemize}
	\item[in \textbf{Case 2}] (see Figure \ref{fig:ninjaparticle2} and \ref{fig:ninjaparticle_block}) if $x_{k+1}=x_i+1$ the transition 
		$$(x_1,\ldots, x_{k+1})\to(x_1,\ldots, x_{i-1}, x_i+2(x_{k+1}-x_i),x_{i+1},\ldots,x_i)  $$
		in the \textit{Ninja}-process, if allowed, corresponds to the transition
		$$ \sum_{i=1}^k \delta_{\pi(x_i,x_{k+1})} \to \sum_{i=1}^k \delta_{\pi(x_i,x_{k+1})}  -\delta_{x_i}+\delta_{x_i+1}$$
		in the $(\mathcal  N^{[N-1]_0}_t)_{t\ge 0}$ process.
		
		\item[In \textbf{Case 3}] (see Figure \ref{Ninja at the boundary}) if 
		${x_k+1}=N$ and $x_i=N-2$,
		i.e. we are in \textbf{Case 3}, 
		$$(x_1,\ldots, x_{k+1}) \to (x_1,\ldots, x_{i-1}, x_{k+1},x_{i+1},\ldots,x_i+(x_{k+1}-x_i)/2)   $$
			in the \textit{Ninja}-process, corresponds to the transition
		
		$$\sum_{i=1}^k \delta_{\pi(x_i,x_{k+1})} \to \sum_{i=1}^k \delta_{\pi(x_i,x_{k+1})} -\delta_{N-2}+\delta_{N-1}$$
	in the 	$(\mathcal  N^{[N-1]_0}_t)_{t\ge 0}$ process and the particle at $N-1$ will stay there forever.
\end{itemize}
Both the above transitions are usual  transitions in the dual BD-SSEP $(\xi^{[N-1]_0}_t)_{t\ge 0}$.
\end{proof}

We conclude this section by collecting in the corollary below some direct consequences of Propositions \ref{proposition: Ninja sep = dual sep} and \ref{proposition: f map} which will be used in the proof of Theorem \ref{theorem: recursion absorption probabilities via Ninja}. The result below follows from the observation that the function
$$f(x_1,\ldots, x_k):= \one_{ \left\{  \left(\sum_{i=1}^k \one_{\{x_i=N\}}\right)=k \right\} }$$
is permutation invariant.
\begin{corollary}\label{corollary: useful results from Ninja}
	\begin{enumerate}
		\item For all  $x_1,\ldots,x_{k+1}\in[N]_0$ with
		$$1\le x_1<\ldots<x_{k}\le N \qquad \text{and} \qquad x_{k+1}\notin\{x_1,\ldots,x_k\},$$  $$\P^{\textit{Ninja}}_{x_1,\ldots,x_k,x_{k+1}}(\{X^{(1)}_\infty=N,\ldots, X_\infty^{(k)}=N,\textit{Ninja}_\infty=N\})\\=\hat\P^{[N]_0}_{\{x_1,\ldots,x_{k+1}\}}(\xi^{[N]_0}_\infty(N)=k+1). $$
		
		\item For each $1\le x_1<\ldots<x_k\le N-1$,
		\begin{multline}\label{corollary: absorb prob on the right}\P^{\textit{Ninja}}_{x_1,\ldots,x_k,x_{k+1}}(\{X^{(1)}_\infty=N,\ldots, X_\infty^{(k)}=N\})\\=	\P^{\textit{Ninja}}_{x_1,\ldots,x_k,y_{k+1}}(\{X^{(1)}_\infty=N,\ldots, X_\infty^{(k)}=N\})= \hat \P^{[N-1]_0}_{\{x_1,\ldots,x_k\}}(\xi^{[N-1]_0}_\infty(N-1)=k).	\end{multline}
		for all $x_{k+1}, y_{k+1}\in [N]_0\setminus\{x_1,\ldots,x_k\}$.

	\end{enumerate}

\end{corollary}

%The second one, it is a simple observation of the fact that once the \textit{Ninja}  is at site $N$ and, let us say, the $j$-th particle arrives at $N-1$ (resp. 1) and wants to jump to the right (resp. to the left), then the systems simply perform a relabelling of the particles where now the $j$-th particles is absorbed at $N$ (resp. at $0$) and the \textit{Ninja}  goes to $N-1$ (resp. to $1$). Thus, overall, our system, behaves as the usual configuration process dual of the boundary driven SSEP except that there is the relabbeling just described: such rellabelling however does not influence the probability that that eventually all the particles are absorbed at $N$, i.e. 

 \subsection{Proof of Theorem \ref{theorem: recursion absorption probabilities via Ninja}}
 
%We start by identifying a martingale for the dynamics described in the previous subsection that will allow us to prove the relation for the absorption probabilities.
 We now have almost all the elements to prove Theorem \ref{theorem: recursion absorption probabilities via Ninja}. Indeed by Corollary \ref{corollary: useful results from Ninja} we have that, for all~$x_{k+1}\notin \{x_1,\ldots,x_k\}$,

 \begin{align*}
 \hat \P^{[N-1]_0}_{\{x_1,\ldots,x_k\}}(\xi^{[N-1]_0}_\infty(N-1)=k) & =		\P^{\textit{Ninja}}_{x_1,\ldots,x_k,x_{k+1}}(\{X^{(1)}_\infty=N,\ldots, X_\infty^{(k)}=N,\textit{Ninja}_\infty=N\})\\
 & \quad +\	\P^{\textit{Ninja}}_{x_1,\ldots,x_k,x_{k+1}}(\{X^{(1)}_\infty=N,\ldots, X_\infty^{(k)}=N,\textit{Ninja}_\infty=0\}).
  \end{align*}

Denoting by $E$ be the event that all the $k$ particles are absorbed at $N$, i.e. 
 
 \begin{align}\label{eq: event E}
 E:=\{X^{(i)}_\infty=N\ \forall i\in\{1,\ldots, k\}\}
 \end{align}
  we obtain, again by Corollary \ref{corollary: useful results from Ninja},

 \begin{multline*}
 \hat \P^{[N-1]_0}_{\{x_1,\ldots,x_k\}}(\xi^{[N-1]_0}_\infty(N-1)=k)=
 \hat\P^{[N]_0}_{\{x_1,\ldots,x_{k+1}\}}(\xi^{[N]_0}_\infty(N)=k+1)\\+\P^{\textit{Ninja}}_{x_1,\ldots,x_k,x_{k+1}}(\textit{Ninja}_\infty=0 | E )\P^{[N-1]_0}_{\{x_1,\ldots,x_k\}}(\xi^{[N-1]_0}_\infty(N-1)=k).
 \end{multline*}

 The conclusion of the proof follows from the next proposition.

 \begin{proposition}
 	 Let $(X^{(1)}_t,\ldots, X^{(k)}_t, \textit{Ninja}_t)_{t\ge 0}$ be the \textit{Ninja}-process and recall the event  $E$ given in \eqref{eq: event E}. Then, for all $x_{k+1}\notin \{x_1,\ldots,x_k\}$,
 	\begin{align}\label{eq: conditional probability}
 	\P^{\textit{Ninja}}_{x_1,\ldots,x_k,x_{k+1}}(\textit{Ninja}_\infty=0 | E )=1 - \frac{x_{k+1}-k}{N}.
 	\end{align}
 \end{proposition}
\begin{proof}
We consider the skeleton chain
$$(X^{(1)}_n,\ldots, X^{(k)}_n, \textit{Ninja}_n)_{n\in\N}$$
since we are interested in computing absorption probabilities only and we define by 
$$(\bar X^{(1)}_n,\ldots, \bar X^{(k)}_n, \overline{\textit{Ninja}}_n)_{n\in\N}$$
the Ninja process conditioned on the event $E$.

In order to compute the conditional probability on the left hand side of \eqref{eq: conditional probability}, we introduce  the following auxiliary stochastic  process:

$$\bar M_n:=N-\overline{\textit{Ninja}}_n +\sum_{\ell=1}^k\1_{ \{\bar X^{(\ell)}_n  <\overline{\textit{Ninja}}_n\}}.$$

First notice that, $\P^{\textit{Ninja}}_{x_1,\ldots,x_k, x_{k+1}} $-a.s.

$$\lim_{n\to \infty}\bar M_n = N\ \one_{\{ \overline{\textit{Ninja}}_\infty=0\}}.$$

Because $\bar M_n\le N+k$ for each $n\in\N$, the dominated convergence theorem guarantees that

$$\lim_{n\to \infty} \E^{\textit{Ninja}}_{x_1,\ldots,x_k, x_{k+1}}  [\bar M_n] = N\  \P^{\textit{Ninja}}_{x_1,\ldots,x_k, x_{k+1}}  (\overline{\textit{Ninja}}_\infty=0  ) = N\  \P^{\textit{Ninja}}_{x_1,\ldots,x_k, x_{k+1}}  (  {\textit{Ninja}}_\infty=0  | E).$$
We conclude by showing that for all $n\in\N$
$$ \E^{\textit{Ninja}}_{x_1,\ldots,x_k, x_{k+1}}  [\bar M_n]=\bar M_0=N-x_{k+1}+k$$
from which the thesis follows.

It is enough to show that for all $x_1,\ldots,x_{k+1}$
$$\E^{\textit{Ninja}}_{x_1,\ldots,x_k, x_{k+1}}[\bar M_1]= N-x_{k+1}+k$$
since then, by the Markov property, we will show that the equality holds for all $n\in\N$.

For this purpose, we consider three different scenarios. The first one is when there exists exactly one $i$ such that $|x_i-x_{k+1}|=1$ or when $x_i=2$ and $x_{k+1}=0$ or when $x_i=N-2$ and $x_{k+1}=N$. In this case, if a particle with label $j\ne i$ jumps, the process will not change its value since each term in the sum composing $\bar M_1$ remains the same as in $\bar M_0$. If the $i$-th particle is at the left of the Ninja (the $k+1$ particle), i.e. $x_{k+1}-x_i=1$, then if it jumps to the left, no terms changes in $\bar M_1$, while if it jumps to the right interacting with the ninja then 

$$\overline{\textit{Ninja}}_1-\overline{\textit{Ninja}}_0=-1$$
and 
$$\1_{\{\bar X^{x_i}_1<\overline{\textit{Ninja}}_1\}}-\1_{\{\bar X^{x_i}_0<\overline{\textit{Ninja}}_0\}}=-1$$
and thus $\bar M_0=\bar M_1$ for any possible transition. Similarly, one can deduce the same conclusions when $x_i-x_{k+1}=1$, or when $x_i=2$ and $x_{k+1}=0$ or when $x_i=N-2$ and $x_{k+1}=N$. 
The second one consists in the case where there exists $i$ and $j$ such that $x_{k+1}-x_i=1$ and $x_{j}-x_{k+1}=1$; then the Ninja cannot move and all the other particles cannot jump across the Ninja and thus $\bar M_1=\bar M_0$.

Finally the third scenario consists in the case
$\min\{|x_i-x_{k+1}|,  \ i\in\{1,\ldots, k\}\} \ge 2$ with $x_{k+1}\notin\{0,N\}$.
In this case, no particle can jump across the \textit{Ninja} in one step and thus 
$$\sum_{i=1}^k\1_{\{\bar X^{(i)}_1<\overline{\textit{Ninja}}_1\}}=\sum_{i=1}^k\1_{\{\bar X^{(i)}_0<\overline{\textit{Ninja}}_0\}}.$$
 However, in this case the Ninja can move both to the left and to the right. We now show that in this third case,  despite the conditioning to the event $E$, 

$$\P^{\textit{Ninja}}_{x_1,\ldots,x_k, x_{k+1}}  (\overline {\textit{Ninja}}_1=x_{k+1}+1  )= \P^{\textit{Ninja}}_{x_1,\ldots,x_k, x_{k+1}}  (\overline {\textit{Ninja}}_1=x_{k+1}-1  )$$
from which we conclude that 
$$\E^{\textit{Ninja}}_{x_1,\ldots,x_k, x_{k+1}}[\bar M_1]=\bar M_0.$$
Recall that by definition of the \textit{Ninja}-process: the non-conditioned process satisfies 
$$\P^{\textit{Ninja}}_{x_1,\ldots,x_k, x_{k+1}}  ( {\textit{Ninja}}_1=x_{k+1}+1  )= \P^{\textit{Ninja}}_{x_1,\ldots,x_k, x_{k+1}}  ( {\textit{Ninja}}_1=x_{k+1}-1  )$$
since in this case, the \textit{Ninja} is performing a simple symmetric random walk. By Bayes theorem and the Markov property we have 

\begin{align*}
&\P^{\textit{Ninja}}_{x_1,\ldots,x_k, x_{k+1}}  (\overline {\textit{Ninja}}_1=x_{k+1}+1  )=\P^{\textit{Ninja}}_{x_1,\ldots,x_k, x_{k+1}}  ({\textit{Ninja}}_1=x_{k+1}+1  | E)\\
&= \frac{\P^{\textit{Ninja}}_{x_1,\ldots,x_k, x_{k+1}}  (E|{\textit{Ninja}}_1=x_{k+1}+1  )}{\P^{\textit{Ninja}}_{x_1,\ldots,x_k, x_{k+1}}  ( E)}\P^{\textit{Ninja}}_{x_1,\ldots,x_k, x_{k+1}}  ( {\textit{Ninja}}_1=x_{k+1}+1  )\\
&= \frac{\P^{\textit{Ninja}}_{x_1,\ldots,x_k, x_{k+1}+1}  (E)}{\P^{\textit{Ninja}}_{x_1,\ldots,x_k, x_{k+1}}  ( E)}\P^{\textit{Ninja}}_{x_1,\ldots,x_k, x_{k+1}}  ( {\textit{Ninja}}_1=x_{k+1}+1  ).
\end{align*}
By Corollary \ref{corollary: absorb prob on the right}, we have that 

$$\P^{\textit{Ninja}}_{x_1,\ldots,x_k, x_{k+1}+1}  (E)=\P^{\textit{Ninja}}_{x_1,\ldots,x_k, x_{k+1}}  ( E)$$
concluding that 
$$\P^{\textit{Ninja}}_{x_1,\ldots,x_k, x_{k+1}}  (\overline {\textit{Ninja}}_1=x_{k+1}+1  )=\P^{\textit{Ninja}}_{x_1,\ldots,x_k, x_{k+1}}  ( {\textit{Ninja}}_1=x_{k+1}+1  ).$$
The same arguments gives 
$$\P^{\textit{Ninja}}_{x_1,\ldots,x_k, x_{k+1}}  (\overline {\textit{Ninja}}_1=x_{k+1}-1  )=\P^{\textit{Ninja}}_{x_1,\ldots,x_k, x_{k+1}}  ( {\textit{Ninja}}_1=x_{k+1}-1  )$$
and thus 
$$\P^{\textit{Ninja}}_{x_1,\ldots,x_k, x_{k+1}}  (\overline {\textit{Ninja}}_1=x_{k+1}+1  )=P^{\textit{Ninja}}_{x_1,\ldots,x_k, x_{k+1}}  (\overline{\textit{Ninja}}_1=x_{k+1}-1  ).$$
Finally, denoting by $(\bar{\mathcal F}_{n})_{n\in\N}$ the natural filtration generated by the conditioned process
\[
(\bar X^{(1)}_n,\ldots, \bar X^{(k)}_n, \overline{\textit{Ninja}}_n)_{n\in\N},
\]
we obtain, using the Markov property,
\begin{align*}
&\E^{\textit{Ninja}}_{x_1,\ldots,x_k, x_{k+1}}[\bar M_n] \\
&\quad=\E^{\textit{Ninja}}_{x_1,\ldots,x_k, x_{k+1}}\Big[\E^{\textit{Ninja}}_{x_1,\ldots,x_k, x_{k+1}}[\bar M_n\mid\bar{\mathcal F}_{n-1}]\Big] \\
&\quad=\E^{\textit{Ninja}}_{x_1,\ldots,x_k, x_{k+1}}\Big[\E^{\textit{Ninja}}_{\bar X^{(1)}_{n-1},\ldots,\bar X^{(k)}_{n-1}, \overline{\textit{Ninja}}_{n-1}}[\bar M_1]\Big].
\end{align*}
The above term is equal to 
\begin{adjustwidth}{-0.45in}{-0.45in}
\begin{align*}
\sum_{y_1,\ldots,y_{k+1}}\P^{\textit{Ninja}}_{x_1,\ldots,x_k, x_{k+1}} & (\bar X^{(1)}_{n-1}=y_1,\ldots, \bar X^{(k)}_{n-1}=y_k,\overline{\textit{Ninja}}_{n-1}=y_{k+1}  )\E^{\textit{Ninja}}_{y_1,\ldots,y_k, y_{k+1}}[\bar M_1]\\
&=\bar M_0\sum_{y_1,\ldots,y_{k+1}}\P^{\textit{Ninja}}_{x_1,\ldots,x_k, x_{k+1}}  (\bar X^{(1)}_{n-1}=y_1,\ldots, \bar X^{(k)}_{n-1}=y_k,\overline{\textit{Ninja}}_{n-1}=y_{k+1}  )=\bar M_0
\end{align*}
\end{adjustwidth}
and the proof is concluded.
\end{proof}

\begin{footnotesize}
	\subsection*{Funding}
	S.F.\ acknowledges financial support from the Engineering and Physical Sciences Research Council of the United Kingdom through the EPSRC Early Career Fellowship EP/V027824/1. J.A. acknowledges financial support from FAPESP through grants 2023/13453-5 and 2025/02707-1
	\subsection*{Acknowledgments}
%	S.F.\ acknowledges financial support from the Engineering and Physical Sciences Research Council of the United Kingdom through the EPSRC Early Career Fellowship EP/V027824/1.
	  S.F.\ and A.G.C.\ thank the  Hausdorff Institute for Mathematics (Bonn) for its hospitality during the Junior Trimester Program \textit{Stochastic modelling in life sciences} funded by the Deutsche Forschungsgemeinschaft (DFG, German Research Foundation) under Germany’s Excellence Strategy - EXC-2047/1 - 390685813. If the authors' staying at HIM has been so pleasant, productive and full of nice workshops and seminars is in great extent thanks to Silke Steinert-Berndt and Tomasz Dobrzeniecki: the authors are extremely grateful to them. The authors thank Simona Villa for making the pictures and  P. Gon{\c{c}}alves, M. Jara, F. Sau and G. Schütz for useful discussions and comments. S.F. thanks P. Gon{\c{c}}alves for her kind hospitality at the  Instituto Superior Técnico in Lisbon during June 2023. This project has received funding from the European Research Council (ERC) under the European Union’s Horizon 2020 research and innovative programme (grant agreement n. 715734). 

\subsection*{Competing interests}
The authors have no competing interests to declare that are relevant to the content of this article.
\end{footnotesize}

\bibliographystyle{plainnat}
\bibliography{Ref}

\end{document}